\documentclass{amsart}
\usepackage[utf8]{inputenc}
\usepackage{amsmath}
\usepackage{amsfonts}
\usepackage{amssymb}
\usepackage{amsthm}
\usepackage{amscd}
\usepackage{float}
\usepackage{tikz}
\usepackage{graphicx}
\usepackage[colorlinks=true]{hyperref}
\hypersetup{urlcolor=blue, citecolor=red}
\usepackage{hyperref}

\newtheorem{theorem}{Theorem}[section]

\newtheorem{lemma}[theorem]{Lemma}
\newtheorem{proposition}[theorem]{Proposition}

\theoremstyle{definition}
\newtheorem{definition}[theorem]{Definition}
\newtheorem{remark}[theorem]{Remark}

\newcommand{\R}{\mathbb{R}}
\newcommand{\N}{\mathbb{N}}
\newcommand{\C}{\mathbb{C}}
\newcommand{\T}{\mathbb{T}}
\newcommand{\Z}{\mathbb{Z}}

\begin{document}

\title[Mass-critical NLS on tori]{Improved global well-posedness for mass-critical nonlinear Schr\"odinger equations on tori}
\author{Robert Schippa}
\address{Korea Institute of Advanced Study, Hoegi-ro 85, Dongdaemun-gu, 02455 Seoul}
\email{rschippa@kias.re.kr}

\begin{abstract}
We show new global well-posedness results for mass-critical nonlinear Schr\"odinger equations on tori in one and two dimensions.
For the quintic nonlinear Schr\"odinger equation on the circle we show global well-\-posed\-ness for initial data in $H^s(\T)$ for $s>\frac{1}{3}$ and $\| u_0 \|_{L^2(\T)} \ll 1$. In two dimensions we show global well-\-po\-sed\-ness on possibly irrational tori for $s>\frac{3}{5}$. In the focusing case we need to consider small $L^2$-norm, whereas in the defocusing case large mass is covered.
\end{abstract}

\maketitle

\section{Introduction}

This note is concerned with the global well-posedness of the mass-critical nonlinear Schr\"odinger equation (NLS) on tori. The one-dimensional torus with period $2 \pi \beta$, $\beta> 1/2$ is denoted by $\T_{\beta} = \R / (2\pi \beta \Z)$. For $d=2$ we denote possibly irrational tori with $\underline{\gamma} \in (1/2,1]^{d-1}$ by $\T^d_{\underline{\gamma}} = \T \times \T_{\gamma_1} \times \ldots \times \T_{\gamma_{d-1}}$. The square torus is denoted by $\T^d$. To simplify notations, for $d=1$ let $\underline{\gamma} = 1$.

\smallskip

The Cauchy problems for the mass-critical NLS read
\begin{equation}
\label{eq:MassCriticalNLS}
\left\{ \begin{array}{cl}
i \partial_t u + \Delta u &= \pm |u|^{\frac{4}{d}} u, \quad (t,x) \in \R \times \T^d_{\underline{\gamma}}, \\
u(0) &= u_0 \in H^s(\T^d_{\underline{\gamma}}).
\end{array} \right.
\end{equation}

The energy is given by
\begin{equation*}
E(u)(t) = \int_{\T^d_{\gamma}} \frac{|\nabla_x u|^2}{2} \pm \frac{d}{4+2d} |u|^{2+ \frac{4}{d}} dx
\end{equation*}
with signs matching the nonlinearity. The equation with $+$-sign on the right hand-side is referred to as \textit{defocusing} due to coercivity of the energy, whereas the equation with $-$-sign is referred to as \textit{focusing}.

\medskip

The low regularity local well-posedness of nonlinear Schr\"odinger equations on tori was thoroughly investigated for the first time by Bourgain \cite{Bourgain1993A}, using discrete Fourier restriction and introducing Fourier restriction norm spaces.

\smallskip

By using the $L^{{2+ \frac{4}{d}}}_{t,x}$-Strichartz estimate
\begin{equation}
\label{eq:CriticalStrichartz}
\| e^{it \Delta} P_N f \|_{L^{{2+ \frac{4}{d}}}_{t,x}([0,1] \times \T^d)} \lesssim_\varepsilon N^\varepsilon \| f \|_{L^2(\T)}
\end{equation}
and multilinear refinements, which result from Galilean invariance, Picard iteration becomes possible in Fourier restriction norm spaces. In this way Bourgain \cite{Bourgain1993A} proved local well-posedness of \eqref{eq:MassCriticalNLS} for $d=1$ and $d=2$ on $\T^2$ for $s>0$.

 It is known since \cite{Bourgain1993A} that the derivative loss in \eqref{eq:CriticalStrichartz} cannot be removed for $d=1$, and Kishimoto \cite{Kishimoto2014} proved that \eqref{eq:MassCriticalNLS} fails to be locally well-posed in $L^2(\T^d)$ with $C^5$-data-to-solution mapping in $d=1$ (resp. $C^3$-data-to-solution mapping in $d=2$); see also \cite{TakaokaTzvetkov2001}. 

The first global result on $\T$ below $H^1(\T)$ was proved by Bourgain \cite{Bourgain2004} using normal form transformations. Using trilinear estimates on frequency dependent times, global well-posedness for $s>\frac{1}{2}-\varepsilon$ with some small $\varepsilon>0$ was established. The key ingredient is a trilinear refinement of the $L^6_{t,x}$-Strichartz estimate for separated frequencies on frequency-dependent time intervals. Whereas Bourgain's estimate was qualitative, the smoothing was recently quantified in \cite{Schippa2023}.

\smallskip

In \cite{DeSilvaPavlovicStaffilaniTzirakis2007} De Silva \emph{et al.} reported global well-posedness in the defocusing case in $H^s(\T)$ for $s$ below $\frac{1}{2}$ in a quantitative sense\footnote{An issue with the analysis of the resonant part was pointed out by Tzirakis \cite{Tzirakis2010}.}. The authors used the $I$-method (see   \cite{CollianderKeelStaffilaniTakaokaTao2002,CollianderKeelStaffilaniTakaokaTao2003}) to show energy decay. 
 With the method having become a canonical tool to show low regularity well-posedness of dispersive equations (see \cite[Chapter~3]{Tao2006}), we presently give only a very brief account for context:

\medskip

Let $0<s<1$. The $I$-method revolves around a smoothing operator $I$, which is smoothing of order $1-s$ for frequencies larger than $N$. For frequencies lower than $N$, $I$ acts like the identity mapping. To introduce more dispersion, we rescale to a large torus $\T^d_\lambda$. The effect becomes more favorable for higher regularities. The relation between $\lambda$ and $N$ is given by
\begin{equation*}
\lambda = N^{\frac{1-s}{s}},
\end{equation*}
which ensures that
\begin{equation*}
\| I f_{\lambda} \|_{H^1(\T^d_\lambda)} \lesssim \| f \|_{H^s(\T^d)}.
\end{equation*}

 Although the energy cannot be used as conserved quantity for $f_\lambda \in H^s(\T^d_\lambda)$, $s<1$, it will be finite for $If_\lambda$. We analyze the $I$-energy $E(I u_\lambda)$, were $u_\lambda$ denotes the solution to \eqref{eq:MassCriticalNLS} with $u_\lambda(0) = f_\lambda$, which serves as an almost conserved quantity. The slow $I$-energy growth allows us to show well-posedness of the rescaled system on a time interval $N^\beta \lambda^\alpha$. However, this existence time depends on the very large period $\lambda = \lambda(N)$. Scaling back to the unit torus, the time of existence is given by $N^\beta \lambda^{\alpha-2}$. Provided that this translates to $N^\kappa$ for some $\kappa > 0$, we can conclude the global well-posedness result by taking $N \to \infty$.

\medskip

The result in \cite{DeSilvaPavlovicStaffilaniTzirakis2007} was further improved by Li--Wu--Xu \cite{LiWuXu2011} who showed global well-posedness of the defocusing equation for $s>\frac{2}{5}$. They combined the $I$-method argument with resonant decompositions. In one dimension we use the same resonant decomposition and rescaling to large tori to increase dispersive effects. Using suitable linear and bilinear Strichartz estimates on time intervals much longer than the unit time scale, we can improve the result due to Li--Wu--Xu \cite{LiWuXu2011} in case of small data $\| u_0 \|_{L^2(\T)}$.  Note that the $L^2$-norm remains invariant under rescaling, and our arguments in one dimension currently do not allow for large data. It is currently not clear how to bring the trilinear refinements from \cite{Schippa2023} into play.

\medskip

Precisely, we show the following small mass global well-posedness result:
\begin{theorem}
\label{thm:ImprovedGWPQuinticNLS}
Let $d=1$. For $s>\frac{1}{3}$ \eqref{eq:MassCriticalNLS} is globally well-posed provided that $\| u_0 \|_{L^2(\T)} \ll 1$.
\end{theorem}


\bigskip

Furthermore, we improve the global well-posedness in two dimensions on arbitrary tori $\T^2_\gamma = \T \times \R / (2 \pi \gamma \Z)$ for $\gamma \in (\frac{1}{2},1]$. In two dimensions the nonlinearity is cubic. On the square torus $\T^2$, Bourgain proved the $L^4_{t,x}$-Strichartz estimate in \cite{Bourgain1993A}
\begin{equation}
\label{eq:L4SquareTorus}
\| P_N e^{it \Delta} u_0 \|_{L^4_{t,x}(\T \times \T^2)} \lesssim \exp \big( \frac{c \log N}{\log \log N} \big) \| u_0 \|_{L^2(\T^2)}.
\end{equation}

Global well-posedness in the defocusing case was proved by De Silva \emph{et al.} \cite{DeSilvaPavlovicStaffilaniTzirakis2007} for $s>\frac{2}{3}$ using the $I$-method, which confirmed an assertion of Bourgain in \cite{Bourgain2004}. De Silva--Pavlović--Staffilani--Tzirakis \cite{DeSilvaPavlovicStaffilaniTzirakis2007} relied on the following bilinear Strichartz estimate on the square torus $\T^2_\lambda$ with large period $\lambda \geq 1$ and separated frequencies $N_2 \ll N_1$:
\begin{equation}
\label{eq:BilinearStrichartzEstimate}
\| e^{it \Delta} P_{N_1} u_1 e^{it \Delta} P_{N_2} u_2 \|_{L^2_{t,x}([0,1] \times \T^2_\lambda)} \lesssim \big( \frac{1}{\lambda} + \frac{N_2}{N_1} \big)^{\frac{1}{2}} \| u_1 \|_{L^2(\T^2_\lambda)} \| u_2 \|_{L^2(\T^2_\lambda)}.
\end{equation}

Using decoupling, Bourgain--Demeter \cite{BourgainDemeter2015} extended the $L^4_{t,x}$-Strichartz estimate \eqref{eq:L4SquareTorus} to arbitrary tori. Moreover, using decoupling arguments, Fan \emph{et al.} \cite{FanStaffilaniWangWilson2018} similarly extended the bilinear Strichartz estimate to arbitrary tori, by which the global well-posedness results for the square torus could be extended to the irrational torus. Trilinear Strichartz estimates in $L_{t,x}^4$ were recently proved in \cite{Schippa2023}.

Very recently, Herr--Kwak \cite{HerrKwak2023} proved the sharp $L^4_{t,x}$-Strichartz estimate
\begin{equation*}
\| P_S e^{it \Delta} u_0 \|_{L^4_{t,x}([0,(\log \# S)^{-1}] \times \T^2)} \lesssim \| u_0 \|_{L^2(\T^2)},
\end{equation*}
where $P_S$ denotes the Fourier projection to a finite set $S \subseteq \Z^2$ with $\# S \geq 2$. This implied global well-posedness on $\T^2$ for $s>0$ under the small mass constraint $\| u_0 \|_{L^2(\T^2)} \ll 1$.

\medskip

This raises two questions:
\begin{itemize}
\item Do the arguments extend to irrational tori?
\item When does global well-posedness hold as well in the case of large mass?
\end{itemize}
Herr--Kwak use a Szemerédi--Trotter argument to show the $L^4_{t,x}$-estimate and extension to irrational tori is currently unclear. Here instead we refine the $I$-method argument originally due to De Silva \emph{et al.} \cite{DeSilvaPavlovicStaffilaniTzirakis2007} and introduce a correction term for the $I$-energy by resonant decomposition. We use linear and bilinear Strichartz estimates, which hold regardless of the ratio of periods. We show the following:
\begin{theorem}
\label{thm:GWPNLS2d}
Let $d=2$. For $s > \frac{3}{5}$  \eqref{eq:MassCriticalNLS} is globally well-posed in the defocusing case. In the focusing case \eqref{eq:MassCriticalNLS} is globally well-posed provided that $\| u_0 \|_{L^2(\T_\gamma^2)} \ll 1$.
\end{theorem}

\medskip

\emph{Outline of the paper.} In Section \ref{section:Preliminaries} we fix notations and introduce the function spaces on arbitrary tori. 
In one dimension we show linear and bilinear Strichartz estimates on the time-scale $\lambda/N$ on large tori in adapted function spaces. In two dimensions we stick to the unit time scale on possibly large tori. In Section \ref{section:AuxiliaryLWP} we show local well-posedness on $\T_\lambda$ on a time-scale comparable to $\lambda/N$ for $s>0$ with small mass constraint. On $\T^2_\lambda$ we show large data well-posedness for $s>0$ on a time scale of length $\lambda^{-\delta}$.

 In Section \ref{section:ModifiedEnergies} we explain how we modify the $I$-energy by resonant decomposition depending on the dimension. In Section \ref{section:EnergyGrowthBounds1d} slow growth of the  modified energy on $\T_\lambda$ is established, and in Section \ref{section:EnergyGrowthBounds2d} the corresponding analysis on $\T^2_\lambda$ is carried out. The proofs of Theorems \ref{thm:ImprovedGWPQuinticNLS} and \ref{thm:GWPNLS2d} are concluded in Section \ref{section:Conclusion}.
 
\section{Preliminaries}
\label{section:Preliminaries}
\subsection{Fourier transform conventions}

Recall that for $\underline{\gamma} \in (1/2,1]^{d-1}$ we set
\begin{equation*}
\T^d_{\underline{\gamma}} = \T \times \T_{\gamma_1} \times \ldots \times \T_{\gamma_{d-1}}.
\end{equation*}
For fixed $\underline{\gamma} \in (1/2,1]^{d-1}$, we consider rescaled tori $\T^d_\lambda := \T_{\lambda} \times \T_{\gamma_1 \lambda} \times \ldots \T_{\gamma_{d-1} \lambda}$ with large period $\lambda \geq 1$. We suppress the notation of $\underline{\gamma}$ for the rescaled tori because it will be fixed at the outset, and the estimates will be independent of $\underline{\gamma} \in (1/2,1]^{d-1}$.

 Let $f: \T^d_\lambda \to \C$ be measurable. The Lebesgue norms are given by
 \begin{equation*}
 \| f \|_{L^p(\T^d_\lambda)} = \big( \int_{\T^d_\lambda} |f(x)|^p dx \big)^{\frac{1}{p}}
 \end{equation*}
 for $1 \leq p < \infty$ with the usual modification for $p = \infty$.

 The Fourier coefficients of $f \in L^2(\T^d_\lambda)$ satisfy $k \in \Z^d_\lambda := \Z / \lambda \times \Z / (\lambda \gamma_1) \times \ldots \Z / (\lambda \gamma_{d-1})$\footnote{To be consistent, we suppress dependence on $\gamma$ likewise on the frequency side.} are given by
\begin{equation*}
\hat{f}(k) = \int_{\T^d_\lambda} e^{-i k \cdot x} f(x) dx,
\end{equation*}
and the Fourier inversion formula reads
\begin{equation*}
f(x) = \frac{1}{(\gamma_1 \ldots \gamma_{d-1}) \lambda^d} \sum_{k \in \Z^d/\lambda} e^{i k \cdot x} \hat{f}(k).
\end{equation*}
The dual $(\T^d_\lambda)^* = \Z^d_{\lambda}$ is endowed with normalized counting measure. Let $\tilde{f}:\Z^d_{\lambda} \to \C$ and $A \subseteq \Z^d/\lambda$. We define
\begin{equation*}
\int_A \tilde{f}(k) (dk)_\lambda = \frac{1}{\lambda^d (\gamma_1\ldots \gamma_{d-1})} \sum_{k \in A} \tilde{f}(k).
\end{equation*}

The Fourier transform conventions are supposed to ensure that Plancherel's theorem holds true without constant depending on $\lambda$:
\begin{equation*}
\| f \|^2_{L^2(\T^d_\lambda)} = \frac{1}{\lambda^d (\gamma_1\ldots \gamma_{d-1})} \sum_{k \in \Z^d_\lambda} |\hat{f}(k)|^2.
\end{equation*}

Sobolev norms on $\T^d_\lambda$ are defined by
\begin{equation*}
\| f \|^2_{H^s(\T^d_\lambda)} = \frac{1}{\lambda^d (\gamma_1\ldots \gamma_{d-1})} \sum_{k \in \Z^d_\lambda} \langle k \rangle^{2s} |\hat{f}(k)|^2.
\end{equation*}
For $A \subseteq \Z^d_\lambda$ we denote with $P_A : L^2(\T^d_\lambda) \to L^2(\T^d_\lambda)$ the Fourier projection to $A$. For $N \in 2^{\N_0}$ we denote with $P_N$ the smooth frequency projection to frequencies of size $N$. For $N=1$, we consider frequencies of size $\lesssim 1$.

\subsection{Function spaces}

In the following we define adapted function spaces, which behave well with sharp time cutoff. This will be very useful in the analysis on ``long" time-intervals. $V^p$-spaces were firstly considered by Wiener \cite{Wiener1979}. The predual space $U^p$ has an atomic structure and comprises a logarithmic refinement of the Fourier restriction space; see \cite{KochTataru2005,KochTataru2007} for first applications.

For an introduction to $U^p$-/$V^p$-spaces we refer to Hadac--Herr--Koch \cite{HadacHerrKoch2009} (see also \cite{HadacHerrKoch2010}). The $Y^s$-spaces introduced below were firstly used in \cite{HerrTataruTzvetkov2011} and the following summary can also be found in \cite{HerrKwak2023}.

\begin{definition}
Let $\mathcal{Z}$ be the collection of finite non-decreasing sequences $\{ t_k \}_{k=0}^K$ in $\R$. We define $V^2$ as the space of all right-continuous functions $u: \R \to \C$ with $\lim_{t \to -\infty} u(t) = 0$ and
\begin{equation*}
\| u \|_{V^2} = \big( \sup_{ \{t_k \}_{k=0}^K \in \mathcal{Z}} \sum_{k=1}^K |u(t_k) - u(t_{k-1})|^2 \big)^{\frac{1}{2}} < \infty.
\end{equation*}
\end{definition}
The following $U^p_{\Delta}$-/$V^p_{\Delta}$-spaces will be useful for intermediate steps. The function spaces, in which we formulate the well-posedness results, are the $Y_{\pm}^s$-spaces defined in Definition \ref{definition:YSpaces}.
\begin{definition}
Let $1 \leq p <\infty$. We define
\begin{equation*}
V^p_{\Delta} = \{ u: \R \times \T^d_\lambda \to \C : e^{-it \Delta} u \in V^p L^2(\T^d_\lambda) \}
\end{equation*}
with norm given by
\begin{equation*}
\| u \|_{V_{\Delta}^p} = \big( \sup_{ \{t_k \}_{k=0}^K \in \mathcal{Z}} \sum_{k=1}^K \| e^{-i t_k \Delta} u(t_k) - e^{-i t_{k-1} \Delta} u(t_{k-1}) \|_{L^2}^p \big)^{\frac{1}{p}} < \infty.
\end{equation*}
\end{definition}
The predual $U^p$-spaces \cite{HadacHerrKoch2009} admit an atomic decomposition:
\begin{definition}
A $U^p_{\Delta}$-atom is a function $f: \R \times \T^d_\lambda \to \C$, which admits the representation
\begin{equation*}
f = \sum_{k=1}^N 1_{[t_{k-1},t_k)} a_k, \quad \sum_{k=1}^N \| a_k \|_{L^2}^p = 1.
\end{equation*}
The $U^p$-norm of $u: \R \times \T^d_\lambda \to \C$ is given by
\begin{equation*}
\| u \|_{U^p_\Delta} = \inf \big( \{ \sum_{k=1}^\infty |\lambda_k| : \, u = \sum_{k=1}^\infty \lambda_k f_k \text{ with } \big( \lambda_k \big) \in \ell^1, \quad f_k \text{ is a } U^p_{\Delta}-\text{atom} \} \big).
\end{equation*}
The convergence is in $L^\infty(\R;L^2)$.
\end{definition}

\begin{remark}[Transfer~principle~for~$U^p_{\Delta}$-spaces]
The atomic structure allows one to transfer estimates like linear or bilinear Strichartz estimates for free solutions to $U^p_{\Delta}$-functions, e.g., the linear Strichartz estimates yield
\begin{equation*}
\| P_N e^{it \Delta} f \|_{L^p_{t,x}(\T \times \T^d_\lambda)} \lesssim_\varepsilon (\lambda N)^\varepsilon \| f \|_{U^p_{\Delta}}
\end{equation*}
for $p = \frac{2(d+2)}{d}$. Likewise, the bilinear Strichartz estimates transfer to functions in $U^2_\Delta$. We refer to \cite[Proposition~2.19]{HadacHerrKoch2009} for details.
\end{remark}

\begin{definition}
\label{definition:YSpaces}
For $s \in \R$, we define $Y_{\pm}^s$ as the space $u: \R \times \T^d_\lambda \to \C$ such that $e^{it |\xi|^2} \widehat{u(t)} (\xi)$ lies in $V^2$ for each $\xi \in \Z^d_{\lambda}$ and
\begin{equation*}
\| u \|_{Y_{\pm}^s} = \big( \frac{1}{\lambda^d} \sum_{\xi \in \Z^d_\lambda} (1+|\xi|^2)^s \| e^{\pm it |\xi|^2} \widehat{u(t)}(\xi) \|^2_{V^2} \big)^{\frac{1}{2}} < \infty.
\end{equation*}
For $T > 0$ and $u:[0,T] \times \T^d_\lambda \to \C$ we define
\begin{equation*}
\| u \|_{Y^s_{\pm,T}} = \big( \frac{1}{\lambda^d} \sum_{\xi \in \Z^d_\lambda} (1+|\xi|^2)^s \| e^{\pm it |\xi|^2} \widehat{u(t)}(\xi) \|^2_{V^2(0,T)} \big)^{\frac{1}{2}}.
\end{equation*}
\end{definition}

Note that $\| u \|_{Y^s_{+,T}} = \| \overline{u} \|_{Y^s_{-,T}}$. We write $Y^s_T = Y^s_{+,T}$ for brevity.

\begin{remark}[Transfer~principle~for~$Y^s$-spaces]
\label{remark:TransferY}
The following embedding is immediate from the definition:
\begin{equation*}
Y^0_T \hookrightarrow V^2_\Delta.
\end{equation*}
Moreover, it holds the following \cite{HadacHerrKoch2009} for $p>2$:
\begin{equation*}
V^2_\Delta \hookrightarrow U^p_{\Delta}.
\end{equation*}
\end{remark}

\begin{proposition}[{\cite{HerrKwak2023}}]
\label{prop:PropertiesYs}
$Y^s$-norms have the following properties:
\begin{itemize}
\item Let $A$, $B$ be disjoint subsets of $\Z^d_{ \lambda}$. For $s \in \R$, it holds
\begin{equation}
\label{eq:DisjointYsProperty}
\| P_{A \cup B} u \|^2_{Y^s} = \| P_A u \|^2_{Y^s} + \| P_B u \|^2_{Y^s}.
\end{equation}
\item For $s \in \R$, time $T>0$, and a function $f \in L^1 H^s$ we have
\begin{equation}
\label{eq:DualityYs}
\| \chi_{[0,T)} \cdot \int_0^t e^{i(t-t') \Delta} f(t') dt' \|_{Y^s} \lesssim \sup_{\substack{v \in Y^{-s}: \\ \| v \|_{Y^{-s}} \leq 1}} \big| \int_0^T \int_{\T_\lambda} \overline{v} f dx dt \big|.
\end{equation}
\end{itemize}
\end{proposition}

\subsection{Linear and bilinear Strichartz estimates}

In the following we consider linear and bilinear Strichartz estimates on $\T^d_\lambda$.

In one dimension, we consider times up to $\lambda/N$, which is the time-scale on which we analyze \eqref{eq:MassCriticalNLS} on $\T_\lambda$. We recall the bilinear Strichartz estimate on $\T$ due to Moyua--Vega \cite{MoyuaVega2008} and Hani \cite{Hani2012} (see \cite[Appendix]{Schippa2023} for a different proof):
\begin{theorem}
Let $N \in 2^{\N_0}$ and $I_1$, $I_2$ be intervals contained in $[-10N,10N]$ with
\begin{equation*}
\inf_{k_i \in I_i} |k_1-k_2| \gtrsim N.
\end{equation*}
Then the following estimate holds:
\begin{equation*}
\| P_{I_1} e^{it \partial_{xx}} f_1 P_{I_2} e^{it \partial_{xx}} f_2 \|_{L^2_t([0,N^{-1}],L_x^2(\T))} \lesssim N^{-\frac{1}{2}} \| f_1 \|_{L^2(\T)} \| f_2 \|_{L^2(\T)}.
\end{equation*}
\end{theorem}

Moreover, we have the following linear Strichartz estimates due to Burq--G\'erard--Tzvetkov \cite{BurqGerardTzvetkov2004} and Guo--Li--Yung \cite{GuoLiYung2021}.
\begin{theorem}
Let $N \in 2^{\N_0}$, and $\varepsilon>0$. Then the following estimates hold:
\begin{align}
\label{eq:FrequencyDependentStrichartz}
\| P_N e^{it \partial_{xx}} f \|_{L^6_t([0,N^{-1}],L_x^6(\T))} &\lesssim \| f \|_{L^2(\T)}, \\
\| P_N e^{it \partial_{xx}} f \|_{L^6_{t,x}(\T^2)} &\lesssim_\varepsilon \log(N)^{2+\varepsilon} \| f \|_{L^2(\T)}.
\end{align}
\end{theorem}

We have the following rescaled bilinear Strichartz estimates:
\begin{proposition}
\label{prop:BilinearStrichartzV2}
Let $M \in 2^{\N_0}$, and $I_1$, $I_2$ be intervals contained in $[-10M,10M]$ with
\begin{equation*}
\inf_{k_i \in I_i} |k_1 - k_2| \gtrsim M.
\end{equation*}
Then the following estimate holds:
\begin{equation}
\label{eq:BilinearStrichartzV2}
\| P_{I_1} u_1 P_{I_2} u_2 \|_{L^2_{t,x}([0,\frac{\lambda}{N}] \times \T_\lambda)} \lesssim (\log(\lambda)+1)^2 \big( \frac{1}{M} + \frac{1}{N} \big)^{\frac{1}{2}} \| P_{I_1} u_1 \|_{Y^0_+} \| P_{I_2} u_2 \|_{Y^0_+}.
\end{equation}
\end{proposition}
\begin{proof}
It suffices to show the following estimate for free solutions $u_i = e^{it \partial_x^2} f_i$:
\begin{equation}
\label{eq:FreeBilinearEstimate}
\| P_{I_1} u_1 P_{I_2} u_2 \|_{L^2_{t,x}([0,\frac{\lambda}{N}] \times \T_\lambda)} \lesssim \big( \frac{1}{M} + \frac{1}{N} \big)^{\frac{1}{2}} \| P_{I_1} f_1 \|_{L^2(\T_\lambda)} \| P_{I_2} f_2 \|_{L^2(\T_\lambda)}.
\end{equation}
Indeed, the above display implies by the transfer principle\footnote{This is a consequence of the atomic representation.} in $U^p$-spaces:
\begin{equation*}
\| P_{I_1} u_1 P_{I_2} u_2 \|_{L^2_{t,x}([0,\frac{\lambda}{N}] \times \T_\lambda)} \lesssim \big( \frac{1}{M} + \frac{1}{N} \big)^{\frac{1}{2}} \| P_{I_1} u_1 \|_{U^2_{\Delta}} \| P_{I_2} u_2 \|_{U^2_{\Delta}}.
\end{equation*}
This we can interpolate (see \cite[Proposition~2.20]{HadacHerrKoch2009} and \cite[Corollary~2.21]{HadacHerrKoch2009}) twice with the estimate
\begin{equation*}
\begin{split}
\| P_{I_1} u_1 P_{I_2} u_2 \|_{L^2_{t,x}([0,\frac{\lambda}{N}] \times \T_\lambda)} &\leq \| P_{I_1} u_1 \|_{L^4_{t,x}([0,\frac{\lambda}{N}] \times \T_\lambda)} \| P_{I_2} u_2 \|_{L^4_{t,x}([0,\frac{\lambda}{N}] \times \T_\lambda)} \\
&\lesssim \lambda^c \| P_{I_1} u_1 \|_{U^4_\Delta} \| P_{I_2} u_2 \|_{U^4_\Delta},
\end{split}
\end{equation*}
which incurs a squared logarithmic factor in $\lambda$. Lastly, the estimate \eqref{eq:FreeBilinearEstimate} follows from rescaling to the unit torus:
\begin{equation*}
\| P_{I_1} u_1 P_{I_2} u_2 \|_{L^2_{t,x}([0,\frac{\lambda}{N}] \times \T_\lambda)} = \lambda^{\frac{3}{2}} \| P_{\lambda I_1} u_{1 \lambda} P_{\lambda I_2} u_{2 \lambda} \|_{L^2_{t,x}([0,\frac{1}{\lambda N}] \times \T)}.
\end{equation*}
Note that $\lambda I_j$ are contained in $[-10 \lambda M, 10 \lambda M]$. If $M \gtrsim N$, we partition the time interval $[0,\frac{1}{\lambda N}]$ into intervals  $J$ of size $\frac{1}{\lambda M}$ on which we can use the bilinear Strichartz estimate:
\begin{equation*}
\begin{split}
\| P_{\lambda I_1} u_{1 \lambda} P_{\lambda I_2} u_{\lambda I_2} \|^2_{L^2_{t,x}([0,\frac{1}{\lambda N}] \times \T)} &= \sum_J \| P_{\lambda I_1} u_{1 \lambda} P_{\lambda I_2} u_{2 \lambda} \|^2_{L^2_{t,x}(J \times \T)} \\
&\lesssim \frac{1}{N \lambda} \| u_{1 \lambda}(0) \|^2_{L^2(\T)} \| u_{2 \lambda}(0) \|^2_{L^2(\T)}.
\end{split}
\end{equation*}
The claim follows then from reversing the scaling.

If $M \lesssim N$, we can use the bilinear estimate directly:
\begin{equation*}
\| P_{\lambda I_1} u_{1 \lambda} P_{\lambda I_2} u_{2 \lambda} \|^2_{L^2_{t,x}([0,\frac{1}{\lambda N}] \times \T)} \lesssim \frac{1}{M \lambda} \| u_{1 \lambda}(0) \|^2_{L^2(\T)} \| u_{2 \lambda}(0) \|^2_{L^2(\T)}.
\end{equation*}
\end{proof}

\begin{remark}
We note the following variant, which is immediate from complex conjugation: Let $I_1$, $I_2$ be intervals contained in $[-10M,10M]$ with
\begin{equation*}
\inf_{k_i \in I_i} |k_1 + k_2| \gtrsim M.
\end{equation*}
Then the following estimate holds:
\begin{equation*}
\| P_{I_1} u_1 P_{I_2} u_2 \|_{L^2_{t,x}([0,\frac{\lambda}{N}] \times \T_\lambda)} \lesssim (\log(\lambda) +1)^2 \big( \frac{1}{M} + \frac{1}{N} \big)^{\frac{1}{2}} \| P_{I_1} u_1 \|_{Y^0_{+,T}} \| P_{I_2} u_2 \|_{Y^0_{-,T}}.
\end{equation*}
\end{remark}

\begin{proposition}
\label{prop:LinearStrichartz}
Let $M \in 2^{\N_0}$. If $M \lesssim N$, the following estimate holds:
\begin{equation}
\label{eq:L6LowFrequenciesV2}
\| P_M u \|_{L^6_{t,x}([0,\frac{\lambda}{N}] \times \T_\lambda)} \lesssim \| u \|_{Y^0_T}.
\end{equation}
If $M \gg N$, the following estimate holds:
\begin{equation}
\label{eq:L6HighFrequenciesV2}
\| P_M u \|_{L^6_{t,x}([0,\frac{\lambda}{N}] \times \T_\lambda)} \lesssim \log(M)^3 \| u \|_{Y^0_T}.
\end{equation}
\end{proposition}
\begin{proof}
For estimate \eqref{eq:L6LowFrequenciesV2} it suffices to prove
\begin{equation*}
\| P_M u \|_{L^6_{t,x}([0,\frac{\lambda}{N}] \times \T_\lambda)} \lesssim \| u \|_{U^6_\Delta}
\end{equation*}
by embedding properties of $V^2_{\Delta} \hookrightarrow U^6_\Delta$. The above display is immediate from the estimates in \cite{BurqGerardTzvetkov2004} and rescaling. The second estimate \eqref{eq:L6HighFrequenciesV2} follows from the $L^6_{t,x}$-estimate on the torus due to Guo--Li--Yung \cite{GuoLiYung2021}.
\end{proof}

We recall the following linear and bilinear Strichartz estimates in adapted function spaces (see \cite{BourgainDemeter2015,FanStaffilaniWangWilson2018}):
\begin{proposition}
Let $\lambda \geq 1$ and $N \in 2^{\N_0}$. The following estimate holds:
\begin{equation}
\label{eq:2dLinearEstimateTransfer}
\| P_N u \|_{L^4_{t,x}(\T \times \T^2_\lambda)} \lesssim_\varepsilon \lambda^\varepsilon \| u \|_{Y^0}.
\end{equation}
If $\lambda \geq N$, then the estimate holds without loss of $\lambda^\varepsilon$.

Let $N_1 \ll N_2$. Then the estimate
\begin{equation}
\label{eq:2dBilinearEstimateTransfer}
\| P_{N_1} u_1 P_{N_2} u_2 \|_{L^2_{t,x}(\T \times \T^2_\lambda)} \lesssim_\varepsilon (\lambda N_2)^\varepsilon \big( \frac{1}{\lambda} + \frac{N_1}{N_2} \big)^{\frac{1}{2}} \| u_1 \|_{Y^0} \| u_2 \|_{Y^0}
\end{equation}
holds.
\end{proposition}
\begin{proof}
The linear estimate with $\lambda \geq N$ is a consequence of the scale-invariant Strichartz estimates on frequency-dependent times due to Burq--Gérard--Tzvetkov \cite{BurqGerardTzvetkov2004} and rescaling. For $N \geq \lambda$ we can use the $L^4_{t,x}$-Strichartz estimates on the (irrational) torus
\begin{equation*}
\| P_N e^{it \Delta} f \|_{L^4_{t,x}(\T \times \T^2_{\underline{\gamma}})} \lesssim_\varepsilon N^\varepsilon \| f \|_{L^2(\T^2_{\underline{\gamma}})}
\end{equation*}
and rescale. This extends to $Y^0$ following Remark \ref{remark:TransferY}. The bilinear estimate \eqref{eq:2dBilinearEstimateTransfer} follows from the bilinear estimate for free solutions:
\begin{equation*}
\| P_{N_1} u_1 P_{N_2} u_2 \|_{L^2_{t,x}(\T \times \T^2_\lambda)} \lesssim_\varepsilon (\lambda N_2)^\varepsilon \big( \frac{1}{\lambda} + \frac{N_1}{N_2} \big)^{\frac{1}{2}} \| u_1(0) \|_{L^2(\T^d_\lambda)} \| u_2(0) \|_{L^2(\T^d_\lambda)}.
\end{equation*}
Interpolation with the $L^4_{t,x}$-Strichartz estimate and an application of the transfer principle conclude the proof.
\end{proof}

\section{Auxiliary local well-posedness results}
\label{section:AuxiliaryLWP}

\subsection{Local well-posedness on the torus on an intermediate time-scale}

In the following we solve the quintic NLS on the rescaled torus
\begin{equation}
\label{eq:RescaledNLS1d}
\left\{ \begin{array}{cl}
i \partial_t u + \partial_{xx} u &= \pm |u|^4 u , \quad (t,x) \in \R \times \T_\lambda, \\
u(0) &= u_0 \in H^s(\T_\lambda).
\end{array} \right.
\end{equation}

With the quintic NLS on the unit torus being locally well-posed in $H^s(\T)$ for times $T \lesssim 1$, the quintic NLS on the rescaled torus is expected to be locally well-posed up to times $\lambda^2$. We solve the Cauchy problem in the following on the time-scale $\lambda / N$. Recall $\lambda \sim N^{\frac{1-s}{s}+\varepsilon}$. Thus, for $s<\frac{1}{2}$, we have $\lambda \gg N$, choosing $\varepsilon > 0$ small enough.

With the $L^6$-Strichartz estimates in adapted function spaces at hand, we can show the following auxiliary well-posedness result:
\begin{proposition}
\label{prop:AuxiliaryWellposedness1d}
For $0<s<\frac{1}{2}$, the Cauchy problem \eqref{eq:RescaledNLS1d} is locally well-posed in $Y^s$ on $[0,T] \times \T_\lambda$ for $T \lesssim \frac{\lambda}{N}$ provided that $\| u_0 \|_{H^s(\T_\lambda)} \ll 1$. The solution satisfies
\begin{equation*}
\| u \|_{Y_T^s} \lesssim \| u_0 \|_{H^s} + \| u \|_{Y_T^s}^5.
\end{equation*}
\end{proposition}
\begin{proof}
We use the property \eqref{eq:DualityYs} from Proposition \ref{prop:PropertiesYs} and Littlewood-Paley theory:
\begin{equation*}
\big\| \chi_{[0,T)} \int_0^t e^{i(t-t') \Delta} (|u|^4 u) (t') dt' \big\|_{Y^s} \lesssim \sup_{v} \big| \int_0^T \int_{\T_\lambda} \sum_{M \geq 1} P_M (|u|^4 u) \overline{P_M v} dx dt \big|.
\end{equation*}
We use an additional Littlewood-Paley decomposition of the factors to write
\begin{equation*}
\begin{split}
&\quad \sum_{M \geq 1} \iint_{[0,T] \times \T_\lambda} P_M (|u|^4 u) \overline{P_M v} dx dt \\ &= \sum_{M \geq 1} \sum_{N_1 \gtrsim N_2 \ldots \gtrsim N_5} \iint_{[0,T] \times \T_\lambda} P_M v P_{N_1} u_1 \ldots P_{N_5} u_5 dx dt.
\end{split}
\end{equation*}
We omit the notation of complex conjugates above and in the following because the estimates do not depend on these.
We use Littlewood-Paley dichotomy to argue
\begin{equation*}
\begin{split}
&\quad \sum_{M \geq 1} \sum_{N_1 \gtrsim N_2 \ldots \gtrsim N_5} \iint_{[0,T] \times \T_\lambda} P_M v P_{N_1} u_1 \ldots P_{N_5} u_5 dx dt \\
 &= \sum_{M \sim N_1 \gtrsim N_2 \ldots \gtrsim N_5} \iint_{[0,T] \times \T_\lambda} P_M v P_{N_1} u_1 \ldots P_{N_5} u_5 dx dt \\
 &\quad + \sum_{M \ll N_1 \sim N_2 \gtrsim N_3 \ldots \gtrsim N_5} \iint P_M v P_{N_1} u_1 \ldots P_{N_5} u_5 dx dt \\
&= (I) + (II).
\end{split}
\end{equation*}

We turn to the estimate of $(I)$. We use almost orthogonality and decompose $P_M v$ and $P_{N_1} u_1$ into frequency intervals of size $N_2$:
\begin{equation*}
\begin{split}
&\quad \iint_{[0,T] \times \T_\lambda} P_M v P_{N_1} u_1 \ldots P_{N_5} u_5 dx dt \\ &= \sum_{I_1,I_2} \iint_{[0,T] \times \T_\lambda} P_{M,I_1} v P_{N,I_2} u_1 P_{N_2} u_2 \ldots P_{N_5} u_5 dx dt.
\end{split}
\end{equation*}
After estimating the above expression, the sum over $I_1$ and $I_2$ can be removed by an application of the Cauchy-Schwarz inequality and property \eqref{eq:DisjointYsProperty} of the $Y^s$-spaces.

\medskip

\textbf{Case }$N_2 \lesssim N$. We use the $L^6$-estimate from \eqref{eq:L6LowFrequenciesV2} to find
\begin{equation*}
\begin{split}
&\quad \iint_{[0,T] \times \T_\lambda} P_{M,I_1} v P_{N_1,I_2} u_1 \ldots P_{N_5} u_5 dx dt \\
&\lesssim \| P_{M,I_1} v \|_{L^6_{t,x}} \| P_{N_1,I_2} u_1 \|_{L^6_{t,x}} \ldots \| P_{N_5} u_5 \|_{L^6_{t,x}} \\
&\lesssim \| P_{M,I_1} v \|_{Y_T^0} \| P_{N_1,I_2} u_1 \|_{Y_T^0} \ldots \| P_{N_5} u_5 \|_{Y_T^0}.
\end{split}
\end{equation*}
Now we can use almost orthogonality to sum over $I_1$ and $I_2$.

\medskip

\textbf{Case }$N_2 \gtrsim N$. In this case we can still use the $L^6_{t,x}$-estimates from \eqref{eq:L6HighFrequenciesV2}. We conclude
\begin{equation*}
\lesssim \log(\lambda N_2)^{18} \| P_{M,I_1} v \|_{Y_T^0} \| P_{N_1} u_1 \|_{Y_T^0} \ldots \| P_{N_5} u_5 \|_{Y_T^0}.
\end{equation*}
Recalling that $\lambda \sim N^{\frac{1-s}{s}+\varepsilon} \lesssim N_2^{\frac{1-s}{s}+\varepsilon}$, we have
\begin{equation*}
\lesssim \big( \frac{1}{s} \log(N_2)\big)^{18} \| P_{M,I_1} v \|_{Y_T^0} \| P_{N_1,I_2} u_1 \|_{Y_T^0} \| P_{N_2} u_2 \|_{Y_T^0} \ldots \| P_{N_5} u_5 \|_{Y_T^0}.
\end{equation*}
In this case we have easy summation to $\lesssim N^{-\delta} \| u \|^5_{Y_T^s}$.

We turn to the estimate of $(II)$: Let $N^* = \max(M,N_3)$. We can again use an almost orthogonal decomposition for $u_{N_1}$ and $u_{N_2}$ to frequency intervals of size $N^*$. Again the $L^6_{t,x}$-estimates will be the cornerstone of the analysis:
\begin{equation*}
\begin{split}
&\quad \iint \sum_{M \ll N_1 \sim N_2 \gtrsim N_3 \gtrsim \ldots \gtrsim N_5 = (*)} P_M v P_{N_1,I_1} u_1 \ldots P_{N_5} u_5 \\
&\leq \sum_{(*)} \| P_M v \|_{L^6_{t,x}} \| P_{N_1,I_1} u_1 \|_{L^6_{t,x}} \| P_{N_2,I_2} u_2 \|_{L^6_{t,x}} \ldots \| P_{N_5} u_5 \|_{L^6_{t,x}}. 
\end{split}
\end{equation*}
If $N^* \lesssim N$, then we can use the estimate \eqref{eq:L6LowFrequenciesV2} to find
\begin{equation*}
\lesssim \sum_{(*)} \| P_M v \|_{Y_T^0} \| P_{N_1,I_1} u_1 \|_{Y_T^0} \ldots \| P_{N_5} u_5 \|_{Y_T^0}.
\end{equation*}
For $M \gtrsim N^\varepsilon$ we find from straight-forward summation the estimate
\begin{equation*}
\lesssim N^{-\delta} \| u \|^5_{Y_T^s}.
\end{equation*}
If $N^* \gg N$, then we can use the $L^6_{t,x}$-estimate \eqref{eq:L6HighFrequenciesV2} to close the estimate.
\end{proof}

To carry out the $I$-method, we introduce notations: Let $m(\xi)$ be a smooth non-negative symbol on $\R^d$, which equals $1$ for $|\xi| \leq 1$ and equals $|\xi|^{-1}$ for $|\xi| \geq 2$. For any $N \geq 1$ and $\alpha \in \R$, let $I_N^\alpha$ be defined by
\begin{equation*}
( I_N^\alpha f ) \widehat (\xi) = m \big( \frac{\xi}{N} \big)^{\alpha} \hat{f}(\xi).
\end{equation*}

We have the following invariance lemma \cite[Lemma~12.1]{CollianderKeelStaffilaniTakaokaTao2004}:
\begin{lemma}
\label{lem:TransferLemma}
Let $\alpha_0 > 0$ and $n \geq 1$. Suppose that $Z$, $X_1$,$\ldots$,$X_n$ are translation invariant Banach spaces and $T$ is a translation invariant $n$-linear operator such that we have the estimate
\begin{equation*}
\| I_1^\alpha T(u_1,\ldots,u_n) \|_{Z} \lesssim \prod_{i=1}^n \| I_1^\alpha u_i \|_{X_i}
\end{equation*}
for all $u_1$,$\ldots$,$u_n$ and all $0 \leq \alpha \leq \alpha_0$. Then the following estimate holds:
\begin{equation*}
\| I_N^\alpha T(u_1,\ldots,u_n) \|_Z \lesssim \prod_{i=1}^n \| I_N^\alpha u_i \|_{X_i}
\end{equation*}
with implicit constant independent of $\alpha$.
\end{lemma}

Proposition \ref{prop:AuxiliaryWellposedness1d} and Lemma \ref{lem:TransferLemma} yield the following result on the $I$-Cauchy problem: 
\begin{proposition}
\label{prop:LWPICauchy}
Let $0<s<1$ and $0<T \sim \lambda/N$. Let $I$ be a smooth radially decreasing interpolant of
\begin{equation*}
(I f ) \widehat (\xi) = 
\begin{cases}
\hat{f}(\xi), \quad &|\xi| \leq N, \\
\big( \frac{N}{|\xi|} \big)^{1-s} \hat{f}(\xi), \quad &|\xi| > 2N.
\end{cases}
\end{equation*}
Then the derived Cauchy problem
\begin{equation}
\label{eq:ICauchyProblem}
\left\{ \begin{array}{cl}
i \partial_t I u + \partial_{xx} I u &= \pm I (|u|^4 u), \quad (t,x) \in [0,T] \times \T_\lambda,\\
Iu(x,0) &= I u_0 \in H^1(\T_\lambda)
\end{array} \right.
\end{equation}
is locally well-posed in $Y_T^1$ provided that $\| u_0 \|_{H^s(\T_\lambda)} \ll 1$.
\end{proposition}

\subsection{Auxiliary local well-posedness in two dimensions}

Next, we show a local well-posedness result on the unit time scale in two dimensions.
\begin{proposition}
\label{prop:AuxiliaryLWP2d}
Let $0 <s < \frac{2}{3}$ and $\lambda = N^{\frac{1-s}{s}}$. Then the $I$-system 
\begin{equation*}
\left\{ \begin{array}{cl}
i \partial_t I u + \Delta I u &= \pm I (|u|^2 u), \quad (t,x) \in \R \times \T^2_\lambda, \\
I u(0) &= I u_0 \in H^1(\T^2_\lambda)
\end{array} \right.
\end{equation*}
is locally well-posed on a time-scale $T \sim \lambda^{-\delta}$ for any $\delta > 0$. This means for any $C>0$ and $\| I u_0 \|_{H^1(\T^2_\lambda)} \leq C$ there exists a solution $I u \in Y^1_T$ which satisfies
\begin{equation*}
\| I u \|_{Y^1_T} \leq \bar{C}(C,\delta)
\end{equation*}
independently of $N \gg 1$ and which depends Lipschitz-continuously on the initial data.
\end{proposition}
\begin{proof}
We invoke the invariance lemma \ref{lem:TransferLemma} to reduce to show local well-posedness
\begin{equation*}
\left\{ \begin{array}{cl}
i \partial_t u + \Delta u &= \pm |u|^2 u, \quad (t,x) \in \R \times \T^2_\lambda, \\
u(0) &= u_0 \in H^s(\T^2_\lambda)
\end{array} \right.
\end{equation*}
for $0 < s < 1$. By the property \eqref{eq:DisjointYsProperty} of $Y^s$-spaces and Littlewood-Paley decomposition it suffices to show estimates for ordered dyadic frequency ranges $N_1 \geq N_2 \geq N_3$:
\begin{equation}
\label{eq:2dMultilinearEstimate}
\big| \iint_{[0,T] \times \T^2_\lambda} P_N v P_{N_1} u_1 P_{N_2} u_2 P_{N_3} u_3 dx dt \big| \lesssim_\varepsilon \lambda^{0-} (N_3^*)^\varepsilon (N_4^*)^\varepsilon \| P_N v \|_{Y^0} \prod_{i=1}^3 \| P_{N_i} u_i \|_{Y^0}.
\end{equation}
for $N \sim N_1 \gtrsim N_2 \gtrsim N_3 \gtrsim N_4$ or $N \ll N_1 \sim N_2$. For the proof, decompose the highest and second to highest frequency ranges $N_1^*$, $N_2^*$ into balls of size $N_3^*$ by almost orthogonality. Now, for a function frequency localized in a ball $B= B(x,N_3^*)$ we have the following Strichartz estimates as a consequence of Galilean invariance for $ 4 \leq q \leq \infty$:
\begin{equation*}
\| P_{B} u \|_{L^q_{t,x}([0,T] \times \T^2_\gamma)} \lesssim_{\varepsilon'} (\lambda N_3^*)^{\varepsilon'} (N_3^*)^{4 \big( \frac{1}{4} - \frac{1}{q} \big)} \| u \|_{Y^0}.
\end{equation*}
This follows from interpolating the estimates at $q=4$ and $q=\infty$.

Consequently, an application of H\"older's inequality yields for $0 \leq T \lesssim \lambda^{-\delta}$:
\begin{equation*}
\| P_B u \|_{L^4_{t,x}([0,T] \times \T^2_\gamma)} \lesssim_{\varepsilon'} \lambda^{-\delta \big( \frac{1}{4} - \frac{1}{q} \big)+ \varepsilon'} (N_3^*)^{4 \big( \frac{1}{4} - \frac{1}{q} \big) + \varepsilon'} \| P_B u \|_{Y^0}.
\end{equation*}
For given $\varepsilon > 0$, choose $q > 4$ small enough and $\varepsilon'$ small enough, such that by an application of H\"older's inequality, we find
\begin{equation*}
\big| \iint_{[0,T] \times \T^2_\lambda} P_N v P_{N_1} u_1 P_{N_2} u_2 P_{N_3} u_3 dx dt \big| \lesssim_\varepsilon \lambda^{0-} (N_3^*)^{\varepsilon} \| P_N v \|_{Y^0} \prod_{i=1}^3 \| P_{N_i} u_i \|_{Y^0}.
\end{equation*}

The estimate suffices to show local well-posedness for $s=2 \varepsilon$ by a standard application of the contraction mapping principle. The details are omitted.
\end{proof}

\section{Modified energies}
\label{section:ModifiedEnergies}

\subsection{Modified energies in one dimension}

With the local well-posedness on the time-scale $\lambda/N$ ensured, we next analyze the growth of the modified energy. We use the resonant decomposition due to Li--Wu--Xu \cite{LiWuXu2011} and recall the (necessary) notations. Let $n \in 2 \N$ and we define $n$-multipliers $M_n(k_1,\ldots,k_n)$ on
\begin{equation*}
\Gamma_n = \{ \underline{k} = (k_1,\ldots,k_n) \in \R^n : k_1 + \ldots + k_n = 0 \}.
\end{equation*}
We define
\begin{equation*}
\Lambda_n(M_n;f_1,\ldots,f_n) = \int_{\Gamma_n} M_n(k_1,\ldots,k_n) \prod_{j=1}^n \hat{f}_j(t,k_j) (d\Gamma_n)_\lambda \text{ with } (d \Gamma_n)_\lambda = \prod_{j=1}^{n-1} (dk_j)_\lambda
\end{equation*}
and let $\Lambda_n(M_n) = \Lambda_n(M_n;u,\overline{u},\ldots,u,\overline{u})$. We find by substituting the equation \eqref{eq:MassCriticalNLS} for $d=1$:
\begin{equation*}
\frac{d}{dt} \Lambda_n(M_n) = \Lambda_n(M_n \alpha_n) + i \Lambda_{n+4} \big( \sum_{j=1}^n (-1)^j X_j(M_n) \big),
\end{equation*}
where
\begin{equation*}
\alpha_n = i \sum_{j=1}^n (-1)^j k_j^2; \quad X_j(M_n) = M_n(k_1,\ldots,k_{j-1},k_j+\ldots+k_{j+4},k_{j+5},\ldots,k_{n+4}).
\end{equation*}
We define the $I$-energy by
\begin{equation*}
E_I^1(u(t)) = E(Iu) = \frac{1}{2} \| \partial_x I u \|^2_2 \pm \frac{1}{6} \| Iu(t) \|^6_{L^6} = \Lambda_2(\sigma_2) \pm \Lambda_6(\sigma_6),
\end{equation*}
where
\begin{equation*}
\sigma_2 = - \frac{1}{2} m(k_1) k_1 m(k_2) k_2; \quad \sigma_6 = \frac{1}{6} m(k_1) \ldots m(k_6).
\end{equation*}

We compute by $\alpha_2 = 0$ due to convolution constraint:
\begin{equation}
\label{eq:FirstPartM6}
\frac{d}{dt} \Lambda_2(\sigma_2) = i \Lambda_6 \big( (-1) X_1(\sigma_2) + X_2(\sigma_2) \big) = \frac{i}{6} \sum_{j=1}^6 (-1)^{j+1} m^2(k_j) k_j^2.
\end{equation}
Secondly,
\begin{equation}
\label{eq:MixedNonlinearity}
\frac{d}{dt} \Lambda_6(\sigma_6) = \Lambda_6(\sigma_6 \alpha_6) \pm \underbrace{i \Lambda_{10} \big( \sum_{j=1}^6 (-1)^j X_j(\sigma_6) \big)}_{M^1_{10}}.
\end{equation}
We denote the sum of \eqref{eq:FirstPartM6} and the first term from \eqref{eq:MixedNonlinearity} as $M_6$. In the following we focus on the defocusing case to simplify notations. Since we establish estimates of the modulus of the involved terms, the same estimates work in the focusing case.

The keypoint is to show that the estimates from Li--Wu--Xu hold on the much longer time interval $T \sim \lambda /N$ compared to $T \lesssim 1$. With the energy decay being the same, we can improve on the regularity from Li--Wu--Xu \cite{LiWuXu2011} under small mass assumption.

Let 
\begin{equation*}
\Upsilon = \{ (k_1,\ldots,k_6) \in \Gamma_6 : |k_1^*| \sim |k_2^*| \gtrsim N \}
\end{equation*}
and denote with $\chi_X$ the characteristic function for some set $X$.

We decompose the multiplier $M_6^1$ in resonant $\tilde{A}$ and non-resonant $\tilde{A}^c$ region, which will be defined in Subsection \ref{subsection:ResonantRegion1d}, after using the cancellation by conservation of energy:
\begin{equation*}
\begin{split}
M_6(k_1,\ldots,k_6) &= \frac{i}{6} \sum_{j=1}^6 (-1)^{j+1} m^2(k_j) k_j^2 + \sigma_6 \alpha_6 \\
&= (M_6^1 + M_6^2) (\chi_{\Upsilon} + \chi_{\Upsilon^c}) = (M_6^1 + M_6^2) \chi_{\Upsilon} \\
&= \underbrace{M_6^2 \chi_{\Upsilon} + M_6^1 \chi_{\tilde{A}^c} \chi_{\Upsilon}}_{\tilde{M}_6} + \underbrace{M_6^1 \chi_{\tilde{A}} \chi_{\Upsilon}}_{\overline{M}_6}.
\end{split}
\end{equation*}
For $\tilde{M}_6$ we have a favorable estimate after integration by parts. But this cannot be carried out for $\overline{M}_6$. We define the modified $I$-energy $E_I^2(u(t))$ by
\begin{equation*}
E_I^2(u(t)) = E_I^1(u(t)) + \Lambda_6(\tilde{\sigma}_6), \quad \tilde{\sigma}_6 = - \tilde{M}_6 / \alpha_6.
\end{equation*}
Now it holds
\begin{equation*}
\frac{d}{dt} E_I^2(u(t)) = \Lambda_6(\overline{M}_6) + \Lambda_{10}(\overline{M}_{10}),
\end{equation*}
where
\begin{equation*}
\overline{M}_{10} = i \sum_{j=1}^6 (-1)^j (X_j(\sigma_6) + X_j(\tilde{\sigma}_6)).
\end{equation*}
Below we use the following formula:
\begin{equation}
\label{eq:EnergyIdentity}
E_I^1(u(t)) = E_I^1(u(0)) - \big[ \Lambda_6(\tilde{\sigma}_6) \big]_0^t + \int_0^t \Lambda_6(\overline{M}_6) + \Lambda_{10}(\overline{M}_{10}) ds.
\end{equation}

In the following we focus on the proof of the following estimates for $0<t \lesssim \frac{\lambda}{N}$ and $\frac{1}{3} < s< \frac{1}{2}$:
\begin{align}
\label{eq:EnergyEstimateResonant}
\big| \int_0^t \int \overline{M}_6(k_1,\ldots,k_6) \hat{u}(s,k_1) \ldots \hat{\overline{u}}(s,k_6) d\Gamma_6 ds \big| &\lesssim N^{-3+} \| I u \|^6_{Y_{T}^1}, \\
\label{eq:EnergyEstimateNonresonant}
\big| \int_0^t \int \overline{M}_{10}(k_1,\ldots,k_{10}) \hat{u}(s,k_1) \ldots \hat{\overline{u}}(s,k_{10}) d\Gamma_{10} ds \big| &\lesssim N^{-3+} \| I u \|^{10}_{Y_{T}^1}, \\
\label{eq:BoundaryTermEstimate}
\big| \Lambda_6(\tilde{\sigma}_6)(t) \big| &\lesssim N^{-\delta} \| I u(t) \|^6_{H^1(\T_\lambda)}.
\end{align}

The first and second estimate translate to favorable energy estimates on the extended time-interval of length $\frac{\lambda}{N}$. Since $|\sigma_6|, |\tilde{\sigma}_6| \lesssim 1$, the proof of the second and third estimate is more straight-forward. The bulk of the work is to estimate the resonant part. Moreover, we need to establish size and regularity estimates of the multilinear Fourier multipliers. The size estimates are essentially due to Li--Wu--Xu \cite{LiWuXu2011}. However, the regularity was not analyzed in \cite{LiWuXu2011}, which will be done through Fourier series argument in the forthcoming sections. For this reason also the analysis of the size is repeated.

\subsection{Modified energies in two dimensions}

We turn to the analysis in two dimensions.
The energy is given by
\begin{equation*}
E^1(u)(t) = \int_{\T^2_\lambda} \frac{|\nabla_x u|^2}{2} dx \pm \int_{\T^2_\lambda} \frac{|u|^4}{4} dx,
\end{equation*}
and we decompose the $I$-energy $E_I = E^1(I u) = \Lambda_2(\sigma_2) \pm \Lambda_4(\sigma_4) $ with
\begin{equation*}
\sigma_2 = - \frac{1}{2} k_1 m(k_1) \cdot k_2 m(k_2), \quad \sigma_4 = \frac{1}{4} m(k_1) \ldots m(k_4).
\end{equation*}
Then we compute from substituting \eqref{eq:MassCriticalNLS} and symmetrization
\begin{equation*}
\frac{d}{dt} \Lambda_2(\sigma_2) = \Lambda_4(M_4), \quad M_4 = \frac{i}{4} (-k_1^2 m^2(k_1) + k_2^2 m^2(k_2) - k_3^2 m^2(k_3) + k_4^2 m^2(k_4)).
\end{equation*}
Secondly,
\begin{equation*}
\frac{d}{dt} \Lambda_4(\sigma_4) = \frac{i}{4} m(k_1) \ldots m(k_4) (-k_1^2 + k_2^2 - k_3^2 + k_4^2) \pm i \sum_{j=1}^4 X_j(\sigma_4).
\end{equation*}
We carry out a very simple resonant decomposition, using a correction term for the first term of the above display and some part of $M_4$. The aim is to remain with the estimate $|\overline{M}_4| \lesssim m(N_1^*) N_1^* m(N_3^*) N_3^*$ for the remainder term $\overline{M}_4$ of $M_4$.

 To this end, we suppose that $|k_i| \sim N_i$. If $N_2, N_4 \ll N_1 \sim N_3$ or $N_1,N_3 \ll N_2 \sim N_4$,
 then clearly the resonance function
 \begin{equation*}
 \Omega_4(k_1,k_2,k_3,k_4) = k_1^2 - k_2^2 + k_3^2 - k_4^2, \quad (k_1,\ldots,k_4) \in \Gamma_4
 \end{equation*}
 satisfies the size estimate $|\Omega_4(k_1,\ldots,k_4)| \gtrsim (N_1^*)^2$. We define the relevant region (since the multipliers become trivial in the complement due to energy conservation and convolution constraint):
 \begin{equation*}
 \Upsilon = \{ (k_1,\ldots,k_4) \in \Gamma_4 : \, |k_i| \sim N_i, \; N_1^* \sim N_2^* \gtrsim N \}.
 \end{equation*}
 We define the non-resonant region by
 \begin{equation*}
 \tilde{A} = \{ (k_1,\ldots,k_4) \in \Upsilon : \, |k_1| \sim |k_3| \gg N_3^* \text{ or } |k_2| \sim |k_4| \gg N_3^* \}.
 \end{equation*}
 
With the non-resonant region given by $\tilde{A}$ we define the correction term 
\begin{equation*}
\tilde{\sigma}_4 = \mp \sigma_4 - \underbrace{i \frac{k_1^2 m^2(k_1) - m^2(k_2) k_2^2 + m^2(k_3) k_3^2 - m^2(k_4) k_4^2}{\Omega_4(k_1,\ldots,k_4)} \chi_{\tilde{A}}}_{\sigma_{4,2}}.
\end{equation*}

We define the modified $I$-energy by
\begin{equation*}
E_I^2 = E_I^1 + \Lambda_4(\tilde{\sigma}_4),
\end{equation*} 
and it follows
\begin{equation*}
\frac{d}{dt} E_I^2 = \Lambda_4(\overline{M}_4) + \Lambda_6(\overline{M}_6)
\end{equation*}
with
\begin{equation*}
\overline{M}_4 = M_4 \cdot \chi_{\Upsilon \cap \tilde{A}^c} \quad \text{ and } \quad \overline{M}_6 = \pm i \sum_{j=1}^4 ( X_j(\sigma_4) + X_j(\sigma_{4,2})).
\end{equation*}

We rewrite the formula for $E_I^1$ by the fundamental theorem of calculus as
\begin{equation*}
E_I^1(t) = E_I^1(0) + \int_0^t \Lambda_4(\overline{M}_4) ds + \int_0^t \Lambda_6(\overline{M}_6) ds - [\Lambda_4(\tilde{\sigma}_4)(t) - \Lambda_4(\tilde{\sigma}_4)(0) ].
\end{equation*}
We shall establish the following estimates for $\frac{1}{2} < s \leq \frac{2}{3}$ and $0 \leq T \lesssim \lambda^{-\delta}$:
\begin{align}
\label{eq:2dBoundM4}
\big| \int_0^T \Lambda_4(\overline{M}_4) ds \big| &\lesssim N^{-1+} \lambda^{-\frac{1}{2}}  \| I u \|^4_{Y^1_T}, \\
\label{eq:2dBoundM6}
\big| \int_0^T \Lambda_6(\overline{M}_6) ds \big| &\lesssim N^{-2+} \| I u \|^6_{Y^1_T}, \\
\label{eq:2dBoundaryTerm}
\big| \Lambda_4(\tilde{\sigma}_4)(t) \big| &\lesssim N^{-\delta} \| I u(t) \|^4_{H^1(\T^2_\lambda)}.
\end{align}

The main part is to show estimates for $\Lambda_4(\overline{M}_4)$. The estimates for $\overline{M}_6$ and the boundary term are simpler because the multiplier is bounded.

\section{Energy growth bounds in one dimension}
\label{section:EnergyGrowthBounds1d}
\subsection{Estimate of the resonant contribution}
\label{subsection:ResonantRegion1d}
\subsubsection{Resonant decomposition}
\label{subsubsection:ResonantDecomposition}
In the following we shall estimate the multiplier
\begin{equation*}
M_6(k_1,\ldots,k_6) = \sum_{i=1}^6 (-1)^{i+1} m^2(k_i) k_i^2
\end{equation*}
on $\Gamma_6$ after dyadic localization $|k_i| \sim N_i \in 2^{N_0}$ against the resonance function
\begin{equation*}
\Omega_6(k_1,\ldots,k_6)= \sum_{i=1}^6 (-1)^{i+1} k_i^2.
\end{equation*}
By symmetry we suppose in the following 
\begin{equation}
\label{eq:SymmetryDyadicLocalization}
N_1 \geq N_3 \geq N_5, \, N_2 \geq N_4 \geq N_6, \text{ and } N_1 \geq N_2.
\end{equation}
 By Littlewood-Paley dichotomy we have $N_1^* \sim N_2^*$. In the following we carry out possibly finer decompositions than on the dyadic scale. We shall only consider additional decompositions into intervals with tractable summation.
Consider
\begin{equation*}
A = \{ \underline{k} \in \Gamma_6 : |\Omega_6| \lesssim |M_6| \}.
\end{equation*}
We shall find $\tilde{A} \supseteq A$ with $\tilde{A}$ given as union of products of intervals and let
\begin{equation*}
\overline{M}_6 = M_6 \chi_{\tilde{A}} = \sum_{I_1,\ldots,I_6} M_6 \chi_{I_1}(k_1) \ldots \chi_{I_6}(k_6).
\end{equation*}
Here $\chi_I$ denotes the indicator function for the interval $I$.
The contribution of $\overline{M}_6$ will then be estimated with size and regularity of $M_6$ on $I_1,\ldots,I_6$ (see Proposition \ref{prop:SizeRegularityMBar}) and linear and bilinear Strichartz estimates.

It is easy to see that the following cases are non-resonant:
\begin{lemma}
\label{lem:EasyNonResonance}
Suppose that $|k_i| \sim N_i$, $i=1,\ldots,6$, and \eqref{eq:SymmetryDyadicLocalization} holds. If 
\begin{equation}
\label{eq:NonResonant1}
N_2 \ll N_1,
\end{equation}
or
\begin{equation}
\label{eq:NonResonant2}
N_1^* \sim N_2^* \gtrsim N_3^* \gg N_4^*,
\end{equation}
then we are in the non-resonant case
\begin{equation}
\label{eq:NonResonantEstimate1}
|M_6(k_1,\ldots,k_6)| \lesssim |\Omega_6(k_1,\ldots,k_6)|.
\end{equation}
\end{lemma}
\begin{proof}
In case of \eqref{eq:NonResonant1}, it is immediate that
\begin{equation*}
|\Omega_6(k_1,\ldots,k_6)| = |k_1^2 +k_3^2 + k_5^2 - (k_2^2 + k_4^2 + k_6^2)| \gtrsim N_1^2 - N_2^2 \gtrsim N_1^2.
\end{equation*}
So, when we consider \eqref{eq:NonResonant2}, we suppose that $\{N_1,N_2,N_3\} = \{N_1^*,N_2^*,N_3^*\}$ or $\{N_1^*,N_2^*,N_3^* \} = \{N_1,N_2,N_4\}$.
We obtain from convolution constraint in the first scenario
\begin{equation*}
\begin{split}
\Omega_6(k_1,\ldots,k_6 ) &= k_1^2 - k_2^2 + k_3^2 + O((N_4^*)^2) = (k_1 - k_2) (k_1 + k_2) + k_3^2 + O((N_4^*)^2) \\
&= (k_2-k_1-k_3)k_3 + O(N_1^* N_4^*) \\
 &= -2k_1 k_3 + O(N_1^* N_4^*).
 \end{split}
\end{equation*}
Hence, the lower bound $|\Omega_6(k_1,\ldots,k_6)| \gtrsim N_1^* N_3^*$ is immediate. In the second scenario $\{N_1^*,N_2^*,N_3^* \} = \{N_1,N_2,N_4 \}$ we can argue symmetrically.

If $N_1^* \sim N_3^*$, the size estimate $|M_6| \lesssim N_1^* N_3^*$ is trivial. If $N_1^* \sim N_2^* \gg N_3^*$ we have $|M_6| \lesssim m(N_1^*) N_1^* N_3^*$ by invoking the mean value theorem to estimate $|m^2(k_1) k_1^2 - m^2(k_2) k_2^2| \lesssim m(N_1^*) N_1^* N_3^*$ (see Proposition \ref{prop:SizeRegularityMBar} (i)).
\end{proof}
The following lemma yields our criterion for non-resonance in case the third highest frequency is much lower than the highest frequency.
\begin{lemma}
\label{lem:NonResonant2HighFrequencies}
Suppose that $|k_i| \sim N_i$, $i=1,\ldots,6$, and \eqref{eq:SymmetryDyadicLocalization} holds. Moreover, suppose that $N_1^* \sim N_2^* \gg N_3^* \sim N_4^*$, and $\{N_1^*,N_2^*\} = \{N_1,N_2\}$, and $k_1$ and $k_2$ are of different signs. If $|k_1+k_2| \sim N_{12} \gg \frac{(N_3^*)^2}{N_1^*}$, then we have $|\Omega_6| \gtrsim N_1^* N_{12} \gtrsim |M_6(k_1,\ldots,k_6)|$.
\end{lemma}
\begin{proof}
We write
\begin{equation*}
\Omega_6 = (k_1-k_2)(k_1+k_2) +k_3^2 -k_4^2 +k_5^2 -k_6^2 = (k_1-k_2)(k_1+k_2) + O((N_3^*)^2).
\end{equation*}
With $|k_1-k_2| \sim N_1^*$, we find that under our assumption $|k_1+k_2| \sim N_{12} \gg \frac{(N_3^*)^2}{N_1^*}$, we have
\begin{equation*}
|\Omega_6| \sim N_{12} N_1^*.
\end{equation*}
For an upper bound on $M_6$ we write (supposing that $k_1 > - k_2$, the other case is symmetric):
\begin{equation*}
\begin{split}
M_6(k_1,\ldots,k_6) &= m^2(k_1) k_1^2 - m^2(k_2) k_2^2 + m^2(k_3) k_3^2 - \ldots - m^2(k_6) k_6^2 \\
&= \int_{-k_2}^{k_1} f'(\xi) d\xi + O(m^2(N_3^*) (N_3^*)^2) \\
&= (k_1+k_2) f'(k_*) + O(m^2(N_3^*) (N_3^*)^2)
\end{split}
\end{equation*}
with $f(\xi) = m^2(\xi) \xi^2$. This yields the claimed estimate by our assumptions on $|k_1+k_2|$.
\end{proof}
Note that the above lemma amounts to introducing a bilinear Fourier localization to describe the non- resonant set. On the other hand, for a superset of the resonant set, where $|k_1+k_2| \lesssim \frac{(N_3^*)^2}{N_1^*}$, we can localize $k_1$ and $k_2$ to intervals $I_1$ and $I_2$ of size $\frac{(N_3^*)^2}{N_1^*}$ which summation can be carried out via almost orthogonality.

Lastly, we consider the case with four comparable high frequencies and the fifth frequency being significantly lower.
\begin{lemma}
\label{lem:NonResonant4HighFrequencies}
Suppose that $N_1^* \sim N_4^* \gg N_5^*$. Then one of the following holds:
\begin{align}
\label{eq:HighFrequencyDistributionI}
\{N_1^*,\ldots,N_4^* \} &= \{N_1,N_2,N_3,N_4\}, \\
\label{eq:HighFrequencyDistributionII}
\{N_1^*,\ldots,N_4^*\}  &= \{N_1,N_2,N_4,N_6\}, \\
\label{eq:HighFrequencyDistributionIII}
\{N_1^*,\ldots,N_4^* \} &= \{N_1,N_2,N_3,N_5 \}.
\end{align}
\begin{itemize}
\item[(i)] If $k_2,k_4,k_6$ in Case \eqref{eq:HighFrequencyDistributionII} or $k_1, k_3, k_5$ in Case \eqref{eq:HighFrequencyDistributionIII} are of the same sign, then it holds $|\Omega_6| \gtrsim (N_1^*)^2$.
\item[(ii)] Suppose that we are in Case \eqref{eq:HighFrequencyDistributionII} and $k_2, k_4, k_6$ are not of the same sign. If $k_i$, $i \in \{2,4,6\}$ and $k_1$ are of the same sign, and $|k_i-k_1| \ll N_1^*$, then $|\Omega_6| \gtrsim (N_1^*)^2$. If $k_i$ and $k_1$ are of different signs, and $|k_i+k_1| \ll N_1^*$, then $|\Omega_6| \gtrsim (N_1^*)^2$.
\item[(iii)] Suppose that we are in Case \eqref{eq:HighFrequencyDistributionIII} and $k_1$,$k_3$,$k_5$ are not of the same sign. If $k_i$, $i \in \{1,2,3\}$ and $k_2$ are of the same sign and $|k_i-k_2| \ll N_1^*$, then $|\Omega_6| \gtrsim (N_1^*)^2$. If $k_i$ and $k_2$ are of different signs and $|k_i+k_2| \ll N_1^*$, then $|\Omega_6| \gtrsim (N_1^*)^2$.
\end{itemize}

\end{lemma}
\begin{proof}
Ad (i): Suppose that we are in Case \eqref{eq:HighFrequencyDistributionI} and $k_2,\,k_4,\,k_6$ are of the same sign. Then we find
\begin{equation}
\label{eq:ResonanceExtension}
\begin{split}
\Omega_6 &= k_1^2 - k_2^2 - k_4^2 - k_6^2 + O((N_5^*)^2) \\
&= (k_2 + k_4 + k_6 + O(N_5^*))^2 - k_2^2 - k_4^2 - k_6^2 + O((N_5^*)^2) \\
&= 2 k_2 k_4 + 2 k_2 k_6 + 2 k_4 k_6 + O (N_1^* N_5^*),
\end{split}
\end{equation}
and hence $|\Omega_6| \gtrsim (N_1^*)^2$. The second statement in (i) follows from the same argument.

Ad (ii) and (iii): Suppose that we are in Case \eqref{eq:HighFrequencyDistributionII}. It holds
\begin{equation*}
\begin{split}
\Omega_6 &= k_1^2 -k_2^2 -k_4^2 - k_6^2 + O((N_5^*)^2) = (k_1-k_2)(k_1+k_2) - k_4^2 -k_6^2 + O((N_5^*)^2) \\
&= (k_1-k_2)(k_1+k_2) - k_4^2 -k_6^2 + O((N_5^*)^2).
\end{split}
\end{equation*}
If $k_1$ and $k_2$ are of the same sign, we have $|k_1+k_2| \sim N_1^*$. If now $|k_1-k_2| \ll N_1^*$, then the lower bound $|\Omega_6| \gtrsim (N_1^*)^2$ is immediate. Similarly, if $k_1$ and $k_2$ are of different signs, we have $|k_1-k_2| \sim N_1^*$. If now $|k_1+k_2| \ll N_1^*$, then again $|\Omega_6| \gtrsim (N_1^*)^2$. The other cases described in (ii) and (iii) follow from obvious modifications of the above argument.
\end{proof}

\subsubsection{Description of the resonant set}

We summarize the regions, which we treat as resonant: Recall that we have dyadically localized $|k_i| \sim N_i$. Write $A_{N_i} = [-2 N_i,- N_i/2] \cup [N_i /2,2N_i]$ for $N_i \in 2^{\N}$ and $A_1 = [-2,2]$.

\medskip

(*): \emph{We have that}
\begin{equation}
\label{eq:ResonantRegion}
N_1 \sim N_2, \; N_1^* \sim N_2^*, \; N_3^* \sim N_4^*
\end{equation}
\emph{and one of the following cases:}
\begin{itemize}
\item[(i)] \emph{if $N_1^* \sim N_2^* \gg N_3^* \sim N_4^*$, then $k_1$ and $k_2$ are of different signs and $|k_1+k_2| \lesssim \frac{(N_3^*)^2}{N_1^*},$ }
\item[(ii)] \emph{ if $N_1^* \sim N_4^* \gg N_5^*$, then it holds one of \eqref{eq:HighFrequencyDistributionI}-\eqref{eq:HighFrequencyDistributionIII}. }

\emph{In case of \eqref{eq:HighFrequencyDistributionII} $k_2$, $k_4$, $k_6$ are not of the same sign. If $k_i$, $i \in \{2,4,6\}$ and $k_1$ are of the same sign, then $|k_i - k_1| \sim N_1^*$. If $k_i$ and $k_1$ are of different signs, then $|k_i+k_1| \sim N_1^*$.}

\emph{In case of \eqref{eq:HighFrequencyDistributionIII} $k_1$, $k_3$, $k_5$ are not of the same sign. If $k_i$, $i \in \{1,3,5\}$ and $k_2$ are of the same sign, then $|k_i-k_2| \sim N_1^*$. If $k_i$ and $k_2$ are of different signs, then $|k_i+k_2| \sim N_1^*$},
\item[(iii)] $N_1^* \sim N_5^*$.
\end{itemize}

\medskip

We clarify that we include the case $N_1^* \sim N_5^*$ because it can be estimated without taking advantage of possibly large resonance function. However, we recall the example by V. Sohinger \cite{Tzirakis2010}, which states that for any $K \in \N$,
\begin{equation*}
(k_1,\ldots,k_6) = K (5,-3,6,-2,1,-7) \text{ satisfies } \sum_{i=1}^6 k_i = \sum_{i=1}^6 (-1)^i k_i^2 = 0.
\end{equation*}
Hence, in case $N_1^* \sim N_5^*$, the resonance function can also vanish.

We define
\begin{equation*}
\overline{M}_6(k_1,\ldots,k_6) = \sum_{\substack{I_1,\ldots,I_6: \\ I_i \subseteq A_{N_i}, \\ (*) \text{ holds}}} M_6(k_1,\ldots,k_6) \chi_{I_1}(k_1) \ldots \chi_{I_6}(k_6).
\end{equation*}

\subsubsection{Estimates in the resonant region}
In the following we establish size and regularity estimates on the multiplier
\begin{equation}
\label{eq:DefMBar}
\overline{M}_6(k_1,\ldots,k_6) = \sum_{i=1}^6 (-1)^{i+1} m^2(k_i) k_i^2
\end{equation}
by extending the symbol from $\Gamma_6$ to $\R^6$. The extended symbol will be denoted by $\overline{M}_6^{\R}$. We take advantage of additional frequency localization such that we only have to show regularity of the derivatives in the sense, which takes into account the localization of the frequencies to intervals and not the absolute size of the frequencies.
By symmetry we can assume
\begin{equation*}
|k_1| \geq |k_3| \geq |k_5| \text{ and } |k_2| \geq |k_4| \geq |k_6|.
\end{equation*}
Hence, we have $\{k_1^*,k_2^*\} = \{k_1,k_2\}$.

We use a Fourier series argument on 
\begin{equation*}
\overline{M}_6(k_1,\ldots,k_6) \chi_{I_1}(k_1) \ldots \chi_{I_6}(k_6)
\end{equation*}
with the frequencies $k_i$ being localized to intervals of length $L_i \in 2^{\Z}$ with absolute size $N_i \in 2^{\N_0}$. The following is well-known. We state a precise version for future reference.
\begin{proposition}
Suppose that there is $\overline{M}^{\R}_6: \R^6 \to \C$ with
\begin{equation*}
\overline{M}^{\R}_6(k_1,\ldots,k_6) \chi_{I_1}(k_1) \ldots \chi_{I_6}(k_6) = \overline{M}_6(k_1,\ldots,k_6) \chi_{I_1}(k_1) \ldots \chi_{I_6}(k_6) \text{ on } \Gamma_6,
\end{equation*}
and there is $\overline{M}_6(I_1,\ldots,I_6) > 0$ such that
\begin{equation}
\label{eq:RegularityMBar}
|\partial^\alpha_{\ell_j} \overline{M}^{\R}_6(\ell_1,\ldots,\ell_6)| \lesssim \frac{\overline{M}_6(I_1,\ldots,I_6)}{L_j^{\alpha}} \text{ for } 0 \leq | \alpha | \leq 15.
\end{equation}

Then we have the Fourier series expansion:
\begin{equation}
\label{eq:FourierSeriesExpansion}
\overline{M}^{\R}_6(k_1,\ldots,k_6) = \frac{1}{L_1 \ldots L_6} \sum_{\xi_1,\ldots,\xi_6 \in \Z} m_6(\frac{\xi_1}{L_1},\ldots,\frac{\xi_6}{L_6}) e^{i \frac{k_1 \xi_1}{L_1}} \ldots e^{i \frac{k_6 \xi_6}{L_6}}
\end{equation}
and the Fourier coefficients are given by
\begin{equation}
\label{eq:FourierCoefficients}
m_6(\frac{\xi_1}{L_1},\ldots,\frac{\xi_6}{L_6}) = \int_{\ell_i \in L_i} e^{-i \frac{\xi_1 \ell_1}{L_1}} \ldots e^{- \frac{\xi_6 \ell_6}{L_6}} \overline{M}^{\R}_6(\ell_1,\ldots,\ell_6) d\ell_1 \ldots d \ell_6.
\end{equation}
The coefficients satisfy the estimate
\begin{equation}
\label{eq:EstimateFourierCoefficients}
\big| m_6(\frac{\xi_1}{L_1},\ldots,\frac{\xi_6}{L_6}) \big| \lesssim L_1 \ldots L_6 \overline{M}_6(I_1,\ldots,I_6) \langle \xi_{\max} \rangle^{-15}.
\end{equation}
\end{proposition}

\begin{proof}
For $\xi = 0$, \eqref{eq:EstimateFourierCoefficients} is immediate from \eqref{eq:RegularityMBar} for $\alpha = 0$. For $\xi \neq 0$ we have through integration by parts
\begin{equation*}
\begin{split}
&\quad \int \big( \frac{L_i}{-i \xi_j} \big) \frac{\partial}{\partial \ell_j} \big( e^{ - i \frac{\xi_j \ell_j}{L_j}} \big) \big( \prod_{i \neq j} e^{-i \frac{\xi_i \ell_i}{L_i}} \big) \overline{M}_6(\ell_1,\ldots,\ell_6) d\ell_1 \ldots d\ell_6 \\
&= \frac{L_j}{(-i \xi_j)} \int \big( e^{-i \frac{\xi_1 \ell_1}{L_1}} \ldots e^{-i \frac{\xi_6 \ell_6}{L_6}} \big) \big( \frac{\partial}{\partial \ell_j} \overline{M}_6(\ell_1,\ldots,\ell_6) \big) d\ell_1 \ldots d\ell_6.
\end{split}
\end{equation*}
Hence, \eqref{eq:EstimateFourierCoefficients} is immediate from repeated integration by parts.

\end{proof}

%
%
\begin{remark}
\label{rem:FourierSeriesArgument}
We use the Fourier series expansion in the following way: We want to estimate
\begin{equation*}
\int_{\Gamma_6 \times [0,T]} \overline{M}_6(k_1,\ldots,k_6) \chi_{I_1}(k_1) \ldots \chi_{I_6}(k_6) \hat{u}_1(k_1,t) \ldots \hat{u}_6(k_6,t) d\Gamma_6 dt.
\end{equation*}
We use \eqref{eq:FourierSeriesExpansion} to find
\begin{equation*}
\begin{split}
&\quad \int_{\Gamma_6 \times [0,T]} \overline{M}_6(k_1,\ldots,k_6) \chi_{I_1}(k_1) \ldots \chi_{I_6}(k_6) \hat{u}_1(k_1,t) \ldots \hat{u}_6(k_6,t) d\Gamma_6 dt \\
&= \int_{\Gamma_6 \times [0,T]} \frac{1}{L_1 \ldots L_6} \sum_{\xi_1,\ldots,\xi_6 \in \Z} m_6\big( \frac{\xi_1}{L_1},\ldots,\frac{\xi_6}{L_6} \big)  e^{i \frac{k_1 \xi_1}{L_1}} \ldots e^{i \frac{k_6 \xi_6}{L_6}} \prod_{i=1}^6 \hat{u}_i(k_i,t) d\Gamma_6 dt \\
&= \frac{1}{L_1 \ldots L_6} \sum_{\xi_1,\ldots,\xi_6 \in \Z} m_6(\frac{\xi_1}{L_1},\ldots,\frac{\xi_6}{L_6}) \int_{\T_\lambda \times [0,T]} u_1(x+\frac{\xi_1}{L_1},t) \ldots u_6(x +\frac{\xi_6}{L_6},t) dx dt.
\end{split}
\end{equation*}
At this point we use Strichartz estimates on the product $u_1 \ldots u_6$. Since Strichartz estimates are translation invariant, it suffices to show an estimate
\begin{equation*}
\big| \int_{\T_\lambda \times [0,T]} u_1(x,t) \ldots u_6(x,t) dx dt \big| \lesssim C(I_1,\ldots,I_6) \prod_{i=1}^6 \| u_i \|_{Y^1_T}
\end{equation*}
to conclude by the estimate \eqref{eq:EstimateFourierCoefficients}:
\begin{equation*}
\begin{split}
&\quad \big| \int_{\Gamma_6 \times [0,T]} \overline{M}_6(k_1,\ldots,k_6) \chi_{I_1}(k_1) \ldots \chi_{I_6}(k_6) \hat{u}_1(k_1,t) \ldots \hat{u}_6(k_6,t) d\Gamma_6 dt \big| \\
&\lesssim \overline{M}_6(I_1,\ldots,I_6) C(I_1,\ldots,I_6) \prod_{i=1}^6 \| u_i \|_{Y^1_T}.
\end{split}
\end{equation*}
\end{remark}
%
%

In the following we establish size and regularity estimates for $\overline{M}_6$ by extending the symbol to $\R^6$. We take advantage of additional frequency localization such that we only have to show regularity of the derivatives in the above sense: the required regularity hinges on the possibly smaller frequency localization, compared to the possibly large dyadic frequency. The size estimates are like in \cite{LiWuXu2011}.

\begin{proposition}
\label{prop:SizeRegularityMBar}
Let $I_1,\ldots,I_6$ denote intervals of length $L_i \in 2^{\Z}$. The multiplier $\overline{M}_6(k_1,\ldots,k_6) \chi_{I_1}(k_1) \ldots \chi_{I_6}(k_6)$ defined in \eqref{eq:DefMBar} on $\Gamma_6$ extends trivially to $\R^6$ and satisfies size and regularity assumptions under the following conditions on $I_1,\ldots,I_6$.
\begin{itemize}
\item[(i)] For $|k_i| \sim N_i \in 2^{\N_0}$ and $N_1^* \sim N_3^*$ we have the size estimate and regularity estimates \eqref{eq:RegularityMBar} with
\begin{equation}
\label{eq:SizeEstimateI}
|\overline{M}_6(I_1,\ldots,I_6)| \lesssim m(N_1^*) N_1^* m(N_3^*) N_3^*.
\end{equation}
If $N_1^* \gg N_3^*$ and $|I_1| \sim |I_2| \sim N_3^*$, we have the size and regularity estimate \eqref{eq:RegularityMBar} with
\begin{equation}
\label{eq:SizeEstimateII}
|\overline{M}_6(I_1,\ldots,I_6)| \lesssim m(N_1^*) N_1^* m(N_3^*) N_3^*.
\end{equation} 
\item[(ii)] If $N_1^* \sim |k_1| \sim |k_2| \gtrsim N \gg N_3^* \sim N_4^*$, $|I_1| \sim |I_2| \lesssim \frac{(N_3^*)^2}{N_1^*}$, and $|k_1+k_2| \lesssim \frac{(N_3^*)^2}{N_1^*}$ for $k_i \in I_i$, then the trivial extension satisfies the size and regularity estimate \eqref{eq:RegularityMBar} with
\begin{equation*}
|\overline{M}_6(I_1,\ldots,I_6)| \lesssim (N_3^*)^2.
\end{equation*} 
\item[(iii)] If $|I_j| \lesssim N_5^*$ for $j=1,2,3,4$, and $\max( |k_1+k_2|, |k_3+k_4|) \lesssim N_5^*$, then we have the size and regularity estimate
\begin{equation*}
|\overline{M}_6(I_1,\ldots,I_6)| \lesssim m(N_1^*) N_1^* N_5^*.
\end{equation*}
\item[(iv)] If $|I_j| \lesssim N_{12}$ for $j=1,2,3,4$, and $N_1 \gtrsim N_{12} \sim |k_1+k_2| \sim |k_3+k_4| \gg |k_5^*|$, then the size and regularity estimates \eqref{eq:RegularityMBar} hold with $|\overline{M}_6(I_1,\ldots,I_6)| \lesssim m(N^*_1) N_1^* N_{12}$. 
\end{itemize}
\end{proposition}
\begin{proof}
Case (i): The size and regularity estimates \eqref{eq:SizeEstimateI} in case $N_1^* \sim N_3^*$ are straight-forward. If $N_1^* \gg N_3^*$, the size estimate \eqref{eq:SizeEstimateII} follows from the mean-value theorem and convolution constraint, which yields
\begin{equation*}
|m^2(k_1) k_1^2 - m^2(k_2) k_2^2| \lesssim m^2(N_1^*) N_1^* N_3^*.
\end{equation*}
Lastly,
\begin{equation*}
m^2(N_1^*) N_1^* N_3^* + m^2(N_3^*) (N_3^*)^2 \lesssim m(N_1^*) m(N_3^*) N_1^* N_3^*
\end{equation*}
because $m(k)|k|$ is increasing in $|k|$.

Case (ii): The size estimate is given by
\begin{equation*}
\begin{split}
|m^2(k_1) k_1^2 - m^2(k_2) k_2^2| &\lesssim m^2(N_1^*) |k_1-k_2||k_1 + k_2| \\
&\lesssim m^2(N_1^*) N_1^* \frac{(N_3^*)^2}{N_1^*} \lesssim (N_3^*)^2.
\end{split}
\end{equation*}
and $m^2(k_i) k_i^2 \lesssim (N_3^*)^2$ for $i=3,\ldots,6$. To show the regularity estimate for $k_1$ and $k_2$, we need to take into account the additional frequency localization.

Case (iii): In this case we find like in \cite[Lemma~3.4(iii)]{LiWuXu2011}
\begin{equation*}
\begin{split}
&\quad |m^2(k_1) k_1^2 - m^2(k_2) k_2^2 + \ldots + m^2(k_5) k_5^2 - m^2(k_6) k_6^2| \\
 &\leq |m^2(k_1) k_1^2 - m^2(k_2) k_2^2| + |m^2(k_3) k_3^2 - m^2(k_4) k_4^2| + |m^2(k_5) k_5^2 - m^2(k_6) k_6^2| \\
&\lesssim m^2(N_1^*) N_1^* N_5^* + m^2(N_3^*) N_3^* N_5^* + m^2(N_5^*) (N_5^*)^2 \\
&\lesssim m(N_1^*) N_1^* N_5^* + m(N_1^*) N_1^* m(N_3^*) N_5^* + m(N_1^*) N_1^* m(N_5^*) N_5^* \\
&\lesssim m(N_1^*) N_1^* N_5^*.
\end{split}
\end{equation*}
The central estimate follows from the mean-value theorem, the pen-ultimate estimate from $m(|k|)|k|$ being increasing in $|k|$, and the ultimate estimate from uniform boundedness of $m$.

Case (iv): In this case we can argue like in (iii).
\end{proof}

With the Fourier series expansion at hand, we can show the following estimate for the resonant part:
\begin{proposition}
\label{prop:EnergyGrowthResonant1d}
For any $\frac{1}{3} < s < \frac{1}{2}$, $0<T \lesssim \frac{\lambda}{N}$, we have
\begin{equation*}
\big| \int_0^T \Lambda_6(\overline{M}_6) dt \big| \lesssim N^{-3+} \| I u \|^6_{Y^1_T}.
\end{equation*}
\end{proposition}
\begin{proof}
We carry out a dyadic frequency localization such that $|k_i| \sim N_i \in 2^{\N_0}$. A decreasing rearrangement of $N_1,\ldots,N_6$ is denoted by $N_1^* \geq N_2^* \ldots \geq N_6^*$. Note by LP-dichotomy that $N_1^* \sim N_2^*$. In the resonant case we have $N_1^* \sim N_2^* \gtrsim N_3^* \sim N_4^*$ and by symmetry considerations we can suppose that
\begin{equation*}
N_1 \geq N_3 \geq N_5 \text{ and } N_2 \geq N_4 \geq N_6 \text{ and } N_1 \geq N_2.
\end{equation*}
Moreover, in the resonant case we have $\{ N_1,N_2 \} = \{N_1^*,N_2^* \}$, and we have further restrictions on the signs and separation of $k_i$ if $N_1^* \sim N_4^*$, which we shall recall below.

Since $\overline{M}_6 = 0$ for $|k_1|,\ldots,|k_6| <N$ we suppose that $N_1^* \sim N_2^* \gtrsim N$. Denote $I_T = [0,T]$. It suffices to show the following estimate:
\begin{equation}
\label{eq:KeyEstimateResonantPart}
\big| \int_{\Gamma_6} \int_{I_T} \frac{\overline{M}_6(k_1,\ldots,k_6) \hat{u}_1(k_1,t) \ldots \hat{u}_6(k_6,t)}{\langle k_1 \rangle m(k_1) \ldots \langle k_6 \rangle m(k_6)} d\Gamma_6 dt \big| \lesssim N^{-3+} (N_1^*)^{0-} \prod_{i=1}^6 \| u_i \|_{Y^0_{i,T}}
\end{equation}
with $\| u_i \|_{Y^0_{i,T}} = \| u_i \|_{Y^0_{(-1)^i,T}}$.

We consider the following regions, which take into account the size of $m$:
\begin{equation*}
\begin{split}
A_1 &= \{ \underline{k} \in \Gamma_6 : N_2^* \gtrsim N \gg N_3^* \sim N_4^* \}, \\
A_2 &= \{ \underline{k} \in \Gamma_6 : N_3^* \gtrsim N \gg N_4^* \}, \\
A_3 &= \{ \underline{k} \in \Gamma_6 : N_4^* \gtrsim N \gg N_5^* \}, \\
A_4 &= \{ \underline{k} \in \Gamma_6 : N_5^* \gtrsim N \}.
\end{split}
\end{equation*}
Since $A_2$ is a non-resonant region, we do not estimate this here.

\medskip

$\bullet$ Estimate in $A_1$: We can use almost orthogonality to decompose the support of $k_1$ and $k_2$ into intervals of length $N_3^*$. By Proposition \ref{prop:SizeRegularityMBar} we have $\overline{M}_6(I_1,\ldots,I_6) \lesssim (N_3^*)^2$, which implies by the Fourier series argument (see Remark \ref{rem:FourierSeriesArgument}), H\"older's inequality, and two bilinear Strichartz estimates from Proposition \ref{prop:BilinearStrichartzV2}:
\begin{equation*}
\begin{split}
\eqref{eq:KeyEstimateResonantPart} &\lesssim N^{2s-2} \big| \int_{A_1 \times I_T} \frac{\hat{u}_1(k_1,t_1) \ldots \hat{u}_6(k_6,t)}{|k_1^*|^s |k_2^*|^s \langle k_5^* \rangle \langle k_6^* \rangle} d \Gamma_6 dt \big| \\
&\lesssim N^{-2} N^{2s} (N_1^*)^{-2s} \big| \int_{\T_\lambda \times I_T} u^*_1 u^*_3 \, u_2^* u_4^* \, J_x^{-1} u_5^* J_x^{-1} u_6^* dx dt \big| \\
&\lesssim N^{-2} N^{2s} (N_1^*)^{-2s} \| u_1^* u_3^* \|_{L^2_{t,x}(I_T \times \T_\lambda)} \| u_2^* u_4^* \|_{L^2_{t,x}(I_T \times \T_\lambda)} \| J_x^{-1} u_5^* \|_{L^\infty_{t,x}} \| J_x^{-1} u_6^* \|_{L^\infty_{t,x}} \\
&\lesssim N^{-3+} (N_1^*)^{0-} \prod_{i=1}^6 \| u_i \|_{Y^0_{i,T}}.
\end{split}
\end{equation*}

Hence, the proof is complete in $A_1$.

\medskip

$\bullet$ Estimate in $A_3$: we have
\begin{equation*}
lhs \eqref{eq:KeyEstimateResonantPart} \lesssim N^{4s-4} \big| \int_{A_3 \times I_T} \frac{\overline{M}_6(k_1,\ldots,k_6) \hat{u}_1(k_1,t) \ldots \hat{u}_6(k_6,t)}{|k_1^*|^s |k_2^*|^s |k_3^*|^s |k_4^*|^s \langle k_5^* \rangle \langle k_6^* \rangle} d \Gamma_6 dt \big|.
\end{equation*}
We split $A_3$ into two regions:
\begin{equation*}
\begin{split}
A_{31} &= \{ \underline{k} \in A_3 : N_1^* \gg N_3^* \}, \\
A_{32} &= \{ \underline{k} \in A_3 : N_1^* \sim N_3^* \}.
\end{split}
\end{equation*}
In the first case we can use the same argument like for the estimate of $A_1$, using the size and regularity of $\overline{M}_6$ as stated in Proposition \ref{prop:SizeRegularityMBar}(i), and two bilinear Strichartz estimates by Proposition \ref{prop:BilinearStrichartzV2}.
For the estimate of $A_{32}$ we recall that $N_3^* \sim N_4^*$ in the resonant case, so that
\begin{equation*}
N_1^* \sim N_4^* \text{ in } A_{32}.
\end{equation*}
We recall notations from \cite{LiWuXu2011}: Split $A_{32}$ into three regions
\begin{equation*}
\begin{split}
A_{321} &= \{ \underline{k} \in A_{32} : \{ N_1^*,\ldots,N_4^* \} = \{N_1,\ldots,N_4 \} \}, \\
A_{322} &= \{ \underline{k} \in A_{32} : \{ N_1^*,\ldots,N_4^* \} = \{N_1,N_2,N_3,N_5 \} \}, \\
A_{323} &= \{ \underline{k} \in A_{32} : \{ N_1^*,\ldots,N_4^* \} = \{N_1,N_2,N_4,N_6\} \}.
\end{split}
\end{equation*}
Case 2(a): Estimate in $A_{321}$. We may assume that $k_1>0$ by symmetry. Then we have to consider the following four cases of signs distributed on $k_1,\ldots,k_4$:
\begin{equation*}
\begin{split}
A_{3211} &= \{ \underline{k} \in A_{321} : k_1 >0,k_2>0,k_3<0,k_4<0 \}, \\
A_{3212} &= \{ \underline{k} \in A_{321} : k_1 >0,k_2<0,k_3<0,k_4>0 \}, \\
A_{3213} &= \{ \underline{k} \in A_{321} : k_1 >0,k_2<0,k_3<0,k_4<0 \}, \\
A_{3214} &= \{ \underline{k} \in A_{321} : k_1 >0,k_2<0,k_3>0,k_4<0 \}. 
\end{split}
\end{equation*}
For the estimates of $A_{3211}$,$A_{3212}$,$A_{3213}$ we use bilinear Strichartz estimates, taking advantage of the different directions of propagation.
For $A_{3211}$ and $A_{3212}$ we have by Proposition \ref{prop:SizeRegularityMBar}(i) that $\overline{M}_6(I_1,\ldots,I_6) \lesssim m(N_1^*) m(N_3^*) N_1^* N_3^*$ and can use two bilinear Strichartz estimates on $u_1 u_3$ and $u_2 u_4$. Here we use that
\begin{equation}
\label{eq:SeparationResonantEstimate3211}
|k_1 - k_3| \sim |k_2 - k_4| \gtrsim N_1^*
\end{equation}
and hence, by Proposition \ref{prop:SizeRegularityMBar}:
\begin{equation*}
\begin{split}
lhs \eqref{eq:KeyEstimateResonantPart} &\lesssim N^{2s-2} (N_1^*)^{-2s} \big| \int_{\T_\lambda \times I_T} u_1 u_3 \cdot u_2 u_4 \cdot J_x^{-1} u_5 J_x^{-1} u_6 dx dt \big| \\
&\lesssim N^{2s-2} (N_1^*)^{-2s} \| u_1 u_3 \|_{L^2_{t,x}(I_T \times \T_\lambda)} \| u_2 u_4 \|_{L^2_{t,x}(I_T \times \T_\lambda)} \| J_x^{-1} u_5 \|_{L^\infty_{t,x}} \| J_x^{-1} u_6 \|_{L^\infty_{t,x}} \\
&\lesssim N^{2s-2} (N_1^*)^{-2s} \frac{1}{N} \prod_{i=1}^6 \| u_i \|_{Y^0_{i,T}}.
\end{split}
\end{equation*}
For $A_{3213}$ we observe that
\begin{equation}
\label{eq:SeparationResonantEstimate3213}
|k_1+k_4| \sim |k_2+k_3| = |k_2 + k_3| \gtrsim N_1^* \gtrsim N.
\end{equation}
Hence, we can use the size and regularity estimate $\overline{M}_6 \lesssim m(N_1^*) m(N_3^*) N_1^* N_3^*$, and we can use two bilinear Strichartz estimates on $u_1 u_2$, and $u_2 u_3$. Note that in this case we use Strichartz estimates on pairs of functions, which are estimated in $Y$-spaces of different signs. Hence, we need to check \eqref{eq:SeparationResonantEstimate3213} in contrast to \eqref{eq:SeparationResonantEstimate3211}. We find
\begin{equation*}
\begin{split}
lhs \eqref{eq:KeyEstimateResonantPart} &\lesssim N^{-2+2s} (N_1^*)^{-2s} \big| \int_{\T_\lambda \times I_T} u_1 u_4 \cdot u_2 u_3 \cdot J_x^{-1} u_5 J_x^{-1} u_6 dx dt \big| \\
&\lesssim N^{-2+2s} (N_1^*)^{-2s} \| u_1 u_4 \|_{L^2_{t,x}(I_T \times \T_\lambda)} \| u_2 u_3 \|_{L^2_{t,x}(I_T \times \T_\lambda)} \prod_{i=5}^6 \| J_x^{-1} u_i \|_{L^\infty_{t,x}} \\
&\lesssim N^{-3+} (N_1^*)^{0-} \prod_{i=1}^6 \| u_i \|_{Y^0_{i,T}}.
\end{split}
\end{equation*}
Finally, we turn to $A_{3214}$ and distinguish two cases:
\begin{equation*}
\begin{split}
A_{3214a} &= \{ \underline{k} \in A_{3214} : |k_1 + k_2 | \lesssim |k_5^*| \}, \\
A_{3214b} &= \{ \underline{k} \in A_{3214} : |k_1+k_2| \gg |k_5+k_6 \}.
\end{split}
\end{equation*}

For the estimate in $A_{3214a}$, we observe that
\begin{equation*}
|k_1+k_2| \sim |k_3 + k_4| \lesssim N_5^*.
\end{equation*}
We can localize $k_1$, $k_2$, $k_3$, $k_4$ to intervals of size $N_5^*$ without summation loss by almost orthogonality. Then, we can use the size and regularity estimate from Proposition \ref{prop:SizeRegularityMBar}(iii) to infer that $\overline{M}_6(I_1,\ldots,I_6) \lesssim m(N_1^*) N_1^* N_5^*$ and use one bilinear and three $L^6_{t,x}$-Strichartz estimates to infer:
\begin{equation*}
\begin{split}
lhs \eqref{eq:KeyEstimateResonantPart} &\lesssim N^{3s-3} (N_1^*)^{-3s} \big| \int_{\Gamma_6 \times I_T} \hat{u}_1(k_1,t) \chi_{I_1}(k_1) \ldots \hat{u_6}(k_6,t) \chi_{I_6}(k_6) d \Gamma_6 dt \big| \\
&\lesssim N^{3s-3} (N_1^*)^{-3s} \big| \int_{\T_\lambda \times I_T} u_1 u_5^* \cdot u_2 \cdot u_3 \cdot u_4 \cdot J_x^{-1} u_6^* dx dt \big| \\
&\lesssim N^{3s-3} (N_1^*)^{-3s} \| u_1 u_5^* \|_{L^2_{t,x}(I_T \times \T_\lambda)} \prod_{i=2}^4 \| u_i \|_{L^6_{t,x}} \| J_x^{-1} u_6^* \|_{L^\infty_{t,x}} \\
&\lesssim N^{-3+} (N_1^*)^{0-} N^{-\frac{1}{2}} \prod_{i=1}^6 \| u_i \|_{Y^0_{i,T}}.
\end{split}
\end{equation*}

For the estimate in $A_{3214b}$ we make a dyadic decomposition in $|k_1+k_2| \sim N_{12}$ with $N_5^* \ll N_{12} \lesssim N_1^*$ and by almost orthogonality we localize $k_1$,$k_2$,$k_3$,$k_4$ to intervals $I_1,\ldots,I_4$ of length $N_{12}$. Now we use the size and regularity estimate $\overline{M}_6(I_1,\ldots,I_6) \lesssim m(N_1^*) N_1^* N_{12}$ provided by Proposition \ref{prop:SizeRegularityMBar} and use H\"older's inequality to find
\begin{equation}
\label{eq:A3214bEstimate}
\begin{split}
lhs \eqref{eq:KeyEstimateResonantPart} &\lesssim N^{3s-3} (N_1^*)^{-3s} \big| \int u_1(x,t) \ldots u_6(x,t) dx dt \big| \\
&\lesssim N^{3s-3} N_{12} (N_1^*)^{-3s} \| u_1 u_2 \|_{L^2_{t,x}} \| u_3 u_4 \|_{L^2_{t,x}} \| J_x^{-1} u_5 \|_{L^\infty_{t,x}} \| J_x^{-1} u_6 \|_{L^\infty_{t,x}}.
\end{split}
\end{equation}
We want to apply two bilinear Strichartz estimates.
Note the following: we can write $\| u_1 u_2 \|_{L^2_{t,x}} = \| u_1 \overline{u}_2 \|_{L^2_{t,x}}$ and $u_1$, $\overline{u}_2$ have spatial frequency support in intervals of length $N_{12}$ contained in the positive real numbers. Moreover, the intervals are separated by distance $N_{12}$. By Galilean transform we can use a bilinear Strichartz estimate at frequencies $N_{12}$ and time scale $\frac{\lambda}{N}$ to find
\begin{equation*}
\begin{split}
\| u_1 \overline{u}_2 \|_{L^2_{t,x}} &\lesssim \big( \frac{1}{N_{12}} + \frac{1}{N} \big)^{\frac{1}{2}} \| u_1 \|_{Y^0_{+,T}} \| \overline{u}_2 \|_{Y^0_{+,T}} \\
&\lesssim \big( \frac{1}{N_{12}} + \frac{1}{N} \big)^{\frac{1}{2}} \| u_1 \|_{Y^0_{+,T}} \|u_2 \|_{Y^0_{-,T}}.
\end{split}
\end{equation*}
In a similar spirit we find
\begin{equation*}
\| u_3 u_4 \|_{L^2_{t,x}} \lesssim \big( \frac{1}{N_{12}} + \frac{1}{N} \big)^{\frac{1}{2}} \| u_3 \|_{Y^0_{+,T}} \| u_4 \|_{Y^0_{-,T}}.
\end{equation*}
For $N_{12} \lesssim N$, we have from straight-forward summation
\begin{equation}
\label{eq:A3214bEstimateII}
\eqref{eq:A3214bEstimate} \lesssim N^{-3+} (N_1^*)^{0-} \prod_{i=1}^6 \| u_i \|_{Y^0_{i,T}}.
\end{equation}
If $N \ll N_{12}$, we find after summation over $N_{12} \lesssim N_1^*$ using $s>\frac{1}{3}$
\begin{equation*}
\begin{split}
\eqref{eq:A3214bEstimateII} &\lesssim N^{3s-3} (N_1^*)^{-3s} N_1^* N^{-1} \prod_{i=1}^6 \| u_i \|_{Y^0_{i,T}} \\
&\lesssim N^{-3+} (N_1^*)^{0-} \prod_{i=1}^6 \| u_i \|_{Y^0_{i,T}}.
\end{split}
\end{equation*}

Case 2(b): We turn to the estimate of
\begin{equation*}
A_{322} = \{ \underline{k} \in A_{32} : \{N_1^*,\ldots,N_4^* \} = \{ N_1,N_2,N_3,N_5 \} \}.
\end{equation*}
In the resonant case, $k_1$, $k_3$, and $k_5$ must have different signs. Suppose that $k_1>0$ and $k_2>0$. If $k_3 >0$, $k_5<0$, then we can use two bilinear Strichartz estimates on $u_1 u_2$ and $u_3 u_5$. If $k_1>0$, $k_2>0$, $k_3<0$, $k_5<0$, then we can write
\begin{equation*}
k_1^2 + k_5^2 + k_3^2 -k_2^2 = k_1^2 + k_5^2 + (k_3-k_2)(k_3+k_2). 
\end{equation*}
Hence, in the non-resonant case we have $|k_3+k_2| \gtrsim N_1^*$, which implies that we can use two bilinear Stricharzt estimates on $u_1 u_5$ and $u_2 u_3$.

Case 2(c): Estimate in $A_{323}$. We can apply two bilinear Strichartz estimates for the same reasons like in Case 2(b).

Step 4: Estimate in $A_4$. We can apply the size and regularity estimate of $\overline{M}_6$ provided by Proposition \ref{prop:SizeRegularityMBar}(i) and use six $L^6$-estimates from Proposition \ref{prop:LinearStrichartz}:
\begin{equation*}
\begin{split}
lhs \eqref{eq:KeyEstimateResonantPart} &\lesssim N^{3s-3} \big| \int_{A_4 \times I_T} \frac{\hat{u}_1(k_1,t) \ldots \hat{u}_6(k_6,t)}{(N_2^*)^s (N_4^*)^s (N_5^*)^s m(k_6) \langle k_6^* \rangle} d \Gamma_6 \big| \\
&\lesssim N^{-3+} (N_1^*)^{0-} \| J_x^{0-} u_1 \|_{L^6_{x,t}} \ldots \| J_x^{0-} u_6 \|_{L^6_{x,t}} \lesssim N^{-3+} \prod_{i=1}^6 \| u_i \|_{Y^0_T}.
\end{split}
\end{equation*}

\end{proof}

\subsection{Estimate of the non-resonant contribution}
\label{subsection:NonresonantEstimates}

Next we show the expected size and regularity estimate of $M_6/\Omega_6$ with $|k_i| \sim N_i$ and possibly finer decomposition dictated by the resonance. Recall that we further suppose by symmetry
\begin{equation*}
N_1 \geq N_3 \geq N_5, \quad N_2 \geq N_2 \geq N_4.
\end{equation*}

We show the following:
\begin{lemma}
\label{lem:NonResonantSizeRegularity}
In the non-resonant region $M^1_6 / \Omega_6 : \Gamma_6 \cap \tilde{A}^c \to \C$ satisfies the size and regularity estimate:
\begin{equation*}
\big| \frac{M_6}{\Omega_6}(k_1,\ldots,k_6) \big| \lesssim 1.
\end{equation*}
\end{lemma}

\begin{proof}
This follows from Case-by-Case analysis.

\textbf{Case:} $N_1 \gg N_2$. In this non-resonant case we do not have a finer frequency decomposition than on the dyadic scale: Note that
\begin{equation*}
| \partial_{\xi_i}^\alpha M_6(\xi_1,\ldots,\xi_6)| \lesssim m^2(N_i) N_i^{2-\alpha}.
\end{equation*}
Secondly, note the representation
\begin{equation}
\label{eq:RepresentationDerivativesOmega}
\partial_{\xi_i}^\alpha \Omega_6^{-1} = \sum_{j} \kappa_j \frac{\xi_i^{\gamma_j}}{\Omega_6^{\beta_j}}
\end{equation}
with $2 \beta_j - \gamma_j \geq \alpha + 2$. This follows from induction on $\alpha$. Consequently,
\begin{equation*}
|\partial_{\xi_i}^\alpha \Omega_6^{-1} | \lesssim \Omega_6^{-1} / N_i^\alpha.
\end{equation*}

\textbf{Case:} $N_1 \sim N_2 \gtrsim N_3^* \gg N_4^*$. We suppose by symmetry that $N_3^* = N_3$. We carry out an additional almost orthogonal decomposition of frequency ranges $|k_i| \sim N_i$, $i=1,2$ into intervals of size $N_3^*$ (if $N_3^* \ll N_1^*$). Above we had shown
\begin{equation*}
|\Omega_6(k_1,\ldots,k_6)| \gtrsim N_1 N_3 \text{ for } (k_1,\ldots,k_6) \in \Gamma_6.
\end{equation*}
We choose the following extension from $\Gamma_6$ to $\R^6$:
\begin{equation*}
\begin{split}
\Omega_6(k_1,\ldots,k_6) &= (k_1-k_2)(k_1+k_2) + k_3^2 - k_4^2 + k_5^2 - k_6^2 \\
&= -(k_1-k_2) (k_3+\ldots+k_6) + k_3^2 - k_4^2 +k_5^2 -k_6^2 \\
&= k_3 (k_3 + k_2 - k_1) - (k_1 - k_2)(k_4 + \ldots + k_6) - k_4^2 +k_5^2 -k_6^2 \\
&= -2 k_1 k_3 - k_3 (k_4 + k_5 + k_6) - (k_1 - k_2)(k_4 + \ldots + k_6) - k_4^2 +k_5^2 -k_6^2.
\end{split}
\end{equation*}
This extension denoted by $\tilde{\Omega}_6(k_1,\ldots,k_6)$ satisfies favorable bounds and also the estimates for the derivatives are satisfied.

Secondly, we recall the estimate
\begin{equation*}
|M_6| \lesssim m(N_1^*) N_1^* m(N_3^*) N_3^*.
\end{equation*}
We have
\begin{equation*}
\begin{split}
\big| \partial_{\xi_1} \big( \frac{M_6}{\tilde{\Omega_6}} \big) \big| &\lesssim \big| \frac{\partial_{\xi_1} M_6}{\tilde{\Omega}_6} \big| + \big| \frac{M_6}{\tilde{\Omega}_6^2} \partial_{\xi_1} \tilde{\Omega}_6 \big| \\
&\lesssim \frac{m^2(N_1) N_1}{N_1 N_3} + \frac{m(N_1) m(N_3)}{N_1 N_3} N_3 \lesssim N_3^{-1}.
\end{split}
\end{equation*}
The general claim follows from \eqref{eq:RepresentationDerivativesOmega}. 
Note that in the worst case the derivatives hit $\tilde{\Omega}_6^{-1}$ successively. This gives the bound
\begin{equation*}
\big| \partial_{\xi_1}^\alpha \frac{M_6}{\Omega_6} \big| \lesssim N_3^{-\alpha}.
\end{equation*}

\textbf{Case:} $N_1^* \sim N_2^* \gg N_3^* \sim N_4^*$. Recall that $N_i^* = N_i$ for $i=1,\ldots,4$, $k_1$ and $k_2$ are of different signs and $|k_1+k_2| \gg \frac{(N_3^*)^2}{N_1^*}$. Let $|k_1+k_2| \sim N_{12}$. This allows us an additional almost orthogonal decomposition of frequency ranges $|k_i| \sim N_i$ into intervals of size $N_{12}$. We have
\begin{equation*}
|\Omega_6| \gtrsim N_1 N_{12}, \quad |M_6| \lesssim m^2(N_1^*) N_1^* N_{12} + m^2(N_3^*) (N_3^*)^2.
\end{equation*}
This shows the size estimate. For the regularity estimate for $\xi_1$ and $\xi_2$, we use again \eqref{eq:RepresentationDerivativesOmega} and note that the critical case is when the derivatives hit $\Omega_6^{-k}$ in the denominator: The resulting term has the form
\begin{equation*}
\big| \Omega_6^{-(\alpha+1)} \xi_1^\alpha \big| \lesssim \frac{1}{N_{12}^{\alpha+1} N_1},
\end{equation*}
which is acceptable, since we have reduced the support of $k_1$ and $k_2$ to intervals of size $N_{12}$.

\textbf{Case:} $N_1^* \sim N_4^* \gg N_5^*$. In the non-resonant case we have the negation of $(*)(ii)$. If $k_2$, $k_4$, $k_6$ are of the same sign (or similarly, $k_1$, $k_3$, $k_5$), then we choose the extension described by \eqref{eq:ResonanceExtension}. For the remaining cases arising from the negation of $(*)(ii)$ we choose the trivial extension of $\Omega_6$. Then we have $|\Omega_6| \gtrsim (N_1^*)^2$ and $|M_6| \lesssim m^2(N_1) (N_1^*)^2$. With these at hand and the suitablve extension, it becomes straight-forward to check the size and regularity.
\end{proof}

For the sake of completeness, we record the following trivial size and regularity estimate for $\sigma_6 = m(k_1) \ldots m(k_6)$:
\begin{proposition}
\label{prop:Sigma6SizeRegularity}
We have the following size and regularity estimates for $\sigma_6: \Gamma_6 \to \C$. Suppose that $|k_i| \sim N_i$, then the trivial extension of $\sigma_6$ to $\R^6$ satisfies the size and regularity estimates with:
\begin{equation*}
|\sigma_6(\xi_1,\ldots,\xi_6)| \lesssim 1.
\end{equation*}
\end{proposition}

We obtain in the non-resonant case for $k_i \in I_i$, denoting intervals taking into account possibly finer than dyadic frequency localization:
\begin{equation}
\label{eq:FourierSeriesExpansionNonResonance}
\begin{split}
&\quad \frac{M_6}{\Omega_6}(k_{11} + \ldots+ k_{15},k_2,\ldots,k_6) \\
&= \frac{1}{L_1 \ldots L_6} \sum_{\xi_1,\ldots,\xi_6 \in \Z} m_6(\frac{\xi_1}{L_1},\ldots,\frac{\xi_6}{L_6}) e^{i \frac{(k_{11}+\ldots+k_{15}) \xi_1}{L_1}} \ldots e^{i \frac{k_6 \xi_6}{L_6}}.
\end{split}
\end{equation}

We obtain for the modified energy (here we only consider integration by parts with respect to the first function and disregard possible complex conjugation since the following estimates are independent thereof):
\begin{equation}
\label{eq:ModifiedEnergyIntermediate}
\begin{split}
&\quad \int_0^t \sum_{\substack{k_1 + \ldots + k_6 = 0, \\ k_{11} + k_{12} + \ldots + k_{15} = k_1 }} \frac{M_6}{\Omega_6}(k_{11} + \ldots + k_{15},k_2,\ldots,k_6) \\
&\times \chi_{I_1}(k_{11}+\ldots+k_{15}) \hat{u}(t,k_{11}) \ldots \hat{u}(t,k_{15}) \chi_{I_2}(k_2) \hat{u}(t,k_2) \ldots \chi_{I_6}(k_6) \hat{u}(t,k_6) ds.
\end{split}
\end{equation}
We carry out an additional dyadic decomposition $|k_{1j}| \sim N_{1j}$. Redenoting 
\begin{equation*}
N_{11}, \ldots, N_{15}, N_2, \ldots, N_6 \to K_1,\ldots,K_{10}
\end{equation*}
we have $K_1^* \sim K_2^* \gtrsim N$, where $K_i^*$ denotes a decreasing rearrangement of $K_1,\ldots,K_{10}$.

We plug in the Fourier series expansion \eqref{eq:FourierSeriesExpansion} and Fourier inversion for $\chi_{I_1}$:
\begin{equation}
\label{eq:FourierInversionFrequencyLocalization}
\chi_{I_1}(k_{11} + \ldots + k_{15}) = \int e^{-i(k_{11} + \ldots + k_{15}) \xi} a_{I_1}(\xi) d\xi
\end{equation}
with $\| a_{I_1} \|_{L^1} \lesssim 1$ independent of $I_1$. 

By Fourier inversion in \eqref{eq:ModifiedEnergyIntermediate}, the size and regularity estimates provided by Lemma \ref{lem:NonResonantSizeRegularity}, and \eqref{eq:FourierInversionFrequencyLocalization}, which allows us to consider functions $\hat{u}(k_{11},t), \ldots, \hat{u}(k_{15},t)$ separately, we see it suffices to estimate
\begin{equation}
\big| \int_{0}^T \int_{\T_\lambda} u_1(t,x) \ldots u_{10}(t,x) dx dt \big| \lesssim N^{-3+} \prod_{i=1}^{10} \| m(D) u_i \|_{Y^1_{(i),T}}
\end{equation}
for any choice of $i \in \{ + , - \}$ and where we can assume that 
\begin{equation*}
\text{supp}( \hat{u}_i(t,\cdot)) \subseteq A_{K_i}.
\end{equation*}
Moreover, we suppose by Littlewood-Paley dichotomy that
\begin{equation}
\label{eq:SizeConditionFrequencies} 
K_1^* \sim K_2^* \gtrsim N.
\end{equation}

The key decay estimate is proved in the following proposition:
\begin{proposition}
\label{prop:DecayEstimateNonResonant1d}
Suppose that $0<T\leq \frac{\lambda}{N}$. Then the following estimate holds:
\begin{equation*}
\big| \int_0^T \Lambda_{10}( \overline{M}_{10}) ds \big| \lesssim N^{-3+} \| I u \|_{Y_T^1}^{10}.
\end{equation*}
\end{proposition}
\begin{proof}
With the above reductions, we can focus on the estimate
\begin{equation*}
\big| \int_0^T \int_{\T_\lambda} P_{K_1} u(t,x) \ldots P_{K_{10}} u(t,x) dx ds \big| \lesssim N^{-3+} (K_1^*)^{0-} \prod_{i=1}^{10} \| P_{K_i} I u \|_{Y_{i,T}^1}.
\end{equation*}
To ease notation, suppose that $K_i = K_i^*$.

\medskip

\textbf{Case } $K_1^* \sim K_2^* \gtrsim N$ and $K_3^* \ll N$. Then we can use two bilinear Strichartz estimates on $P_{K_1} u P_{K_3} u$ and $P_{K_2} u P_{K_4} u$ and estimate the remaining terms in $L^\infty_{t,x}$:
\begin{equation*}
\begin{split}
&\quad \big| \int_0^T \int_{\T_\lambda} P_{K_1} u(t,x) \ldots P_{K_{10}} u(t,x) dx ds \big| \\
&\lesssim \| P_{K_1} u P_{K_3} u \|_{L^2_{t,x}} \| P_{K_2} u P_{K_4} u \|_{L^2_{t,x}} \prod_{i=5}^{10} \| P_{K_i} u \|_{L^\infty_{t,x}} \\
&\lesssim \frac{\log(N)^2}{N} \| P_{K_1} u \|_{Y^0_{1,T}} \| P_{K_2} u \|_{Y^0_{2,T}} \prod_{i=3}^{10} K_i^{\frac{1}{2}} \| P_{K_i} u \|_{Y^0_{i,T}}. 
\end{split}
\end{equation*}
Taking the multiplier $m(K_i) K_i$ into account, we find
\begin{equation*}
\lesssim \frac{\log(N)^2}{N^3} K_1^{0-} \prod_{i=1}^{10} \| P_{K_i} I u \|_{Y^1_{i,T}}.
\end{equation*}

\medskip

\textbf{Case} $N_3^* \gtrsim N$ and $N_4^* \ll N$. In this case there are still two bilinear Strichartz estimates applicable such that we find the same estimate like above. This can be further improved properly taking into account the weight $m(K_i) K_i$.

\medskip

\textbf{Case} $N_4^* \gtrsim N$. In this case it suffices to apply six $L^6_{t,x}$-Strichartz estimates and four $L^\infty_{t,x}$-estimates to find
\begin{equation}
\label{eq:StrichartzNonresonant}
\begin{split}
\big| \int_0^T \int_{\T_\lambda} \prod_{i=1}^{10} P_{K_i} u \big| &\lesssim \prod_{i=1}^6 \| P_{K_i} u \|_{L^6_{t,x}} \prod_{i=7}^{10} \| P_{K_i} u \|_{L^\infty_{t,x}} \\
&\lesssim \log(K_1)^{20} (K_7 \ldots K_{10})^{\frac{1}{2}} \prod_{i=1}^{10} \| P_{K_i} u \|_{Y^0_{i,T}} .
\end{split}
\end{equation}
The weight
\begin{equation*}
\frac{1}{m(K_1) K_1} \ldots \frac{1}{m(K_{10}) K_{10}}
\end{equation*}
is least favorable when $K_{10} \gtrsim N$. In this case we have
\begin{equation*}
K_i m(K_i) = \big( \frac{N}{K_i} \big)^{1-s} K_i = N^{1-s} K_i^s
\end{equation*}
and for $\frac{1}{3}<s<1$:
\begin{equation*}
\frac{(K_7 \ldots K_{10})^{\frac{1}{2}}}{m(K_1) K_1 \ldots m(K_{10}) K_{10}} = N^{-4(1-s)} (K_1 \ldots K_4)^{-s} \underbrace{\big( \prod_{i=5}^{10} N^{-(1-s)} K_5^{-s} \big) (K_7 \ldots K_{10})^{\frac{1}{2}} }_{\lesssim 1}.
\end{equation*}
Taking \eqref{eq:StrichartzNonresonant} together with the above display yields
\begin{equation*}
\lesssim K_1^{0-} N^{-4+} \prod_{i=1}^{10} \| P_{K_i} I u \|_{Y^1_{i,T}},
\end{equation*}
which is sufficient.
\end{proof}

Finally, we show the estimate for the boundary term:
\begin{proposition}
\label{prop:BoundaryTerm1d}
Let $\frac{1}{3} < s< \frac{1}{2}$. Then there is $\delta(s)>0$ such that the estimate holds:
\begin{equation*}
\big| \Lambda_6 ( \tilde{\sigma}_6)(t) \big| \lesssim N^{-\delta} \| I u (t) \|_{H^1(\T_\lambda)}^6.
\end{equation*}
\end{proposition}
\begin{proof}
The estimates will be very crude.
By Fourier series expansion, which is possible by Proposition \ref{prop:Sigma6SizeRegularity}, and Littlewood-Paley decomposition it suffices to estimate
\begin{equation*}
\big| \int_{\T_\lambda} P_{N_1} u_1(t,x) \ldots P_{N_6} u_6(t,x) dx \big| \lesssim N_1^{-\delta} \prod_{i=1}^6 N_i \| P_{N_i} m(D) u_i \|_{L^2(\T_\lambda)}
\end{equation*}
for $N_1 \geq N_2 \ldots \geq N_6$, $N_1 \sim N_2 \gtrsim N$. To show the above display, we use H\"older's inequality and obtain
\begin{equation}
\label{eq:BoundaryTerm1}
\begin{split}
&\quad \big| \int_{\T_\lambda} P_{N_1} u_1(t,x) \ldots P_{N_6} u_6(t,x) dx \big| \\
 &\lesssim \| P_{N_1} u_1 \|_{L^2} \| P_{N_2} u_2 \|_{L^2} \prod_{i=3}^6 \| P_{N_i} u_i \|_{L^\infty} \\
&\lesssim (N_3 \ldots N_6)^{\frac{1}{2}} N^{-2(1-s)} N_1^{-2s} (m(N_1) N_1 \| P_{N_1} u_1 \|_{L^2} m(N_2) N_2 \| P_{N_2} u_2 \|_{L^2} ) \prod_{i=3}^6 \| P_{N_i} u_i \|_{L^2}.
\end{split}
\end{equation}

To conclude the proof, we need to take into account the weight $m(N_3) N_3 \ldots m(N_6) N_6$, which is least favorable when $N_6 \gtrsim N$. In this case we still obtain
\begin{equation*}
(m(N_3) N_3 \ldots m(N_6) N_6)^{-1} = N^{-4(1-s)} (N_3^{-s} \ldots N_6^{-s}).
\end{equation*}
Consequently,
\begin{equation}
\label{eq:BoundaryTerm2}
\begin{split}
&\quad (N_3 \ldots N_6)^{\frac{1}{2}} \prod_{i=3}^6 \| P_{N_i} u_i \|_{L^2} \\
 &\lesssim (N_3 \ldots N_6)^{\frac{1}{2}} N^{-4(1-s)} (N_3^{-s} \ldots N_6^{-s}) \prod_{i=3}^6 N_i \| P_{N_i} m(D) u_i \|_{L^2} \\
&\lesssim N^{-4(1-s)} (N_3 \ldots N_6)^{\frac{1}{6}} \prod_{i=3}^6 N_i \| P_{N_i} m(D) u_i \|_{L^2}.
\end{split}
\end{equation}
Since $s>\frac{1}{3}$, plugging \eqref{eq:BoundaryTerm2} into \eqref{eq:BoundaryTerm1} yields
\begin{equation*}
\big| \int_{\T_\lambda} P_{N_1} u_1(t,x) \ldots P_{N_6} u_6(t,x) dx \big| \lesssim N^{-6(1-s)} N_1^{-\delta} \prod_{i=1}^6 N_i \| P_{N_i} m(D) u \|_{L^2(\T_\lambda)}.
\end{equation*}
\end{proof}

\section{Energy growth bounds in two dimensions}
\label{section:EnergyGrowthBounds2d}
\subsection{Estimate of the resonant contribution}

The first step is to show size and regularity estimates for $\overline{M_4}$. The following analysis remains fairly rough. To improve on this, one should look into a refined resonant decomposition of $M_4$.
In the following we suppose dyadic frequency localization
\begin{equation*}
|k_i| \sim N_i, \quad N_1 \geq N_3, \quad N_2 \geq N_4.
\end{equation*}

We show the following straight-forward size and regularity estimate:
\begin{lemma}
\label{lem:SizeRegularityResonantPart2d}
Suppose the frequencies are in dyadic ranges like above and $N_1 \sim N_2$. Then the following size and regularity estimate holds:
\begin{equation}
\label{eq:SizeRegularityResonantPart2d}
|\overline{M_4}| \lesssim m(N_1^*) N_1^* m(N_3^*) N_3^*.
\end{equation}
\end{lemma}
\begin{proof}
In case $N_1^* \sim N_3^*$ the claimed size and regularity estimate is immediate since $m^2(N_i) N_i^2 \lesssim m^2(N_1^*) (N_1^*)^2$ by monotonicity of the function $m(k) k$ for $k \in [1,\infty)$. In case $N_1^* \gg N_3^*$ we firstly decompose the frequency ranges of $k_1$ and $k_2$ (corresponding to $N_1^*$ and $N_2^*$) into balls of size $N_3^*$, which causes no summation loss by almost orthogonality. With $N_1^* \sim N_2^* \sim N_1 \sim N_2$ by our assumptions, we can use the mean-value theorem on $m^2(k_1) k_1^2 - m^2(k_2) (-k_2)^2$ to estimate
\begin{equation*}
|m^2(k_1) k_1^2 - m^2(-k_2) (-k_2)^2| \lesssim | \nabla f(\xi_*)| |k_1+k_2|
\end{equation*}
with $f(\xi) = m^2(\xi) |\xi|^2$. With $\xi_* \in \{ k_1 + t(-k_2-k_1) : t \in [0,1] \}$ the estimate $|\xi_*| \gtrsim N_1$ is evident (because by convolution constraint $|k_1+k_2| \lesssim N_3^* \ll N_1$).

This gives $|\nabla f(\xi_*)| \sim m^2(N_1^*) N_1^*$ and, in conclusion,
\begin{equation*}
|m^2(k_1) k_1^2 - m^2(-k_2) (-k_2)^2 | \lesssim m^2(N_1^*) N_1^* N_3^*.
\end{equation*}
Hence,
\begin{equation*}
|M_4| \lesssim m^2(N_1^*) N_1^* N_3^* + m^2(N_3^*) (N_3^*)^2 \lesssim m(N_1^*) m(N_3^*) N_1^* N_3^*.
\end{equation*}
The regularity of the trivial extension
\begin{equation*}
M_4^{\R}(\xi_1,\ldots,\xi_4) = m^2(\xi_1) \xi_1^2 - m^2(\xi_2) \xi_2^2 + m^2(\xi_3) \xi_3^2 - m^2(\xi_4) \xi_4^2
\end{equation*}
is checked like in Proposition \ref{prop:SizeRegularityMBar}. For the estimates in $\xi_1$ and $\xi_2$ we use the additional localization in Fourier space to argue
\begin{equation*}
|\partial_{\xi_i}^\alpha M_4| \lesssim m(N_1^*) N_1^* m(N_3^*) N_3^* (N_3^*)^{-\alpha} \text{ for } i=1,2.
\end{equation*}
\end{proof}

We are ready to prove the estimate for the energy growth of the resonant part:
\begin{proposition}
\label{prop:EnergyGrowthResonant2d}
Let $0 \leq T \leq \lambda^{-\delta}$ and $\frac{1}{2} < s \leq \frac{2}{3}$. Let $\lambda = N^{\frac{1-s}{s}}$. Then the following estimate holds for any $\varepsilon > 0$
\begin{equation}
\label{eq:EnergyGrowth2dResonant}
\big| \int_0^T \Lambda_4(\overline{M}_4) ds \big| \lesssim_\varepsilon N^{-1+\varepsilon} \lambda^{-\frac{1}{2}} \| I u \|^4_{Y^1_T}.
\end{equation}
\end{proposition}
\begin{proof}
Since the multiplier is otherwise trivial we can suppose that $N_1^* \sim N_2^* \gtrsim N$. By symmetry, we suppose like above $N_1 \sim N_2 \sim N_1^* \sim N_2^* \gg N_3^*$.

\medskip

\textbf{Case A }$N_3^* \ll N$. In this case the size and regularity estimate holds with $|M_4| \lesssim m(N_1) N_1 m(N_3^*) N_3^*$. In the following suppose that $N_3^* = N_3$ to ease notation. After using the Fourier series argument and changing to position space, we can use two bilinear Strichartz estimates (due to large frequency separation):
\begin{equation*}
\begin{split}
&\quad \big| \int_0^T \int_{\Gamma_4} M_4(k_1,\ldots,k_4) \prod_{i=1}^4 \chi_{I_i} \hat{u}_i(k_i) d\Gamma_4 ds \big| \\
&\lesssim m(N_1) N_1 m(N_3^*) N_3^* \big| \int_0^T \int_{\T^2_\gamma} u_1(t,x) \ldots u_4(t,x) dx dt \big| \\
&\lesssim_\varepsilon m(N_1) N_1 m(N_3^*) N_3^* (\lambda N_1)^\varepsilon \big( \frac{1}{\lambda} + \frac{N_3}{N_1} \big)^{\frac{1}{2}} \big( \frac{1}{\lambda} + \frac{N_4}{N_1} \big)^{\frac{1}{2}} \prod_{i=1}^4 \| P_{N_i} u \|_{Y^0_T}.
\end{split}
\end{equation*}
We analyze the expression depending on the size of $N_1$ relative to $N_3$ and $N_4$.

\textbf{Case AI} $N_3 \gtrsim N_1^{\frac{1}{2}}$. Note that for $s \in [\frac{1}{2},\frac{2}{3}]$, $\lambda^{-1} \lesssim N^{-\frac{1}{2}}$ and therefore, by $N \lesssim N_1$ and $N_3 \gtrsim N^{\frac{1}{2}}$, we have
\begin{equation*}
\frac{1}{\lambda} + \frac{N_3}{N_1} \lesssim \frac{1}{N^{\frac{1}{2}}} + \frac{N_3}{N} \lesssim \frac{N_3}{N}.
\end{equation*}

\smallskip

\emph{Subcase AI.i} $N_3 \sim N_4$. In this case we obtain
\begin{equation*}
\begin{split}
&\quad N_1^\varepsilon m(N_1) N_1 m(N_2) N_2 (N^{-(1-s)} N_2^{-s}) N_3 \frac{N_4}{N} \prod_{i=1}^4 \| P_{N_i} u_i \|_{Y^0_T} \\
&\lesssim N_1^{-\varepsilon} N^{-2+} \prod_{i=1}^4 m(N_i) N_i \| P_{N_i} u \|_{Y^0_T}.
\end{split}
\end{equation*}
This is sufficient.

\smallskip

\emph{Subcase AI.ii} $N_4 \ll N_3$. Then the above expression becomes
\begin{equation*}
\begin{split}
&\quad N_1^\varepsilon m(N_1) N_1 m(N_2) N_2 (N^{-(1-s)} N_2^{-s}) N_3 \big( \frac{N_3}{N} \big)^{\frac{1}{2}} \big(\lambda^{-1} + \frac{N_4}{N} \big)^{\frac{1}{2}} \\ &\lesssim N_1^{-\varepsilon} N^{-1+} \lambda^{-\frac{1}{2}} \prod_{i=1}^4 N_i m(N_i).
\end{split}
\end{equation*}
The ultimate step is clear if $\lambda^{-1} \geq N_4 / N$. If $N_4 / N \geq \lambda^{-1}$, we can verify that $(N_4 N)^{-\frac{1}{2}} \leq \lambda^{-\frac{1}{2}}$ for $s \in [\frac{1}{2},\frac{2}{3}]$.

\textbf{Case AII} $N_3 \lesssim N_1^{\frac{1}{2}}$. In this case we have
\begin{equation*}
\frac{1}{\lambda} + \frac{N_i}{N_1} \lesssim \frac{1}{N^{\frac{1}{2}}} \text{ for } i=3,4,
\end{equation*}
and we find
\begin{equation*}
N_1^\varepsilon m(N_1) N_1 m(N_2) N_2 (N^{-(1-s)} N_2^{-s}) N_3 \frac{1}{N^{\frac{1}{4}}} \big( \frac{N_4}{N} + \frac{1}{\lambda} \big)^{\frac{1}{2}} \lesssim N_1^{-\varepsilon} N^{-\frac{5}{4}+} \lambda^{-\frac{1}{2}} \prod_{i=1}^4 m(N_i) N_i,
\end{equation*}
which suffices.

\medskip

\textbf{Case B} $N_3^* \gtrsim N$ and $N_4^* \ll N$. In this case we can use one bilinear Strichartz estimate involving $u_{N_1^*}$ and $u_{N_4^*}$ and two $L^4_{t,x}$-Strichartz estimates to find by the Fourier series argument:
\begin{equation*}
\begin{split}
&\quad \big| \int_0^T \int_{\Gamma_4} \overline{M}_4(k_1,\ldots,k_4) \hat{u}_1(k_1) \ldots \hat{\overline{u}}(k_4) d\Gamma_4 ds \big| \\
&\lesssim m(N_1^*) N_1^* m(N_3^*) N_3^* \big| \int_0^T \int_{\T^2_\gamma} u_1 \ldots u_4(s,x) dx ds \big| \\
&\lesssim_\varepsilon N_1^\varepsilon m(N_1^*) N_1^* m(N_3^*) N_3^* \big( \frac{1}{\lambda} + \frac{N_4}{N_1} \big)^{\frac{1}{2}} \prod_{i=1}^4 \| u_i \|_{Y^0_T} \\
&\lesssim N_1^\varepsilon m(N_1) N_1 m(N_2) N_2 (N^{-(1-s)} N_2^{-s}) m(N_3) N_3 \frac{m(N_4) N_4}{\lambda^{\frac{1}{2}}} \prod_{i=1}^4 \| u_i \|_{Y^0_T} \\
&\lesssim N_1^{-\varepsilon} N^{-1+} \lambda^{-\frac{1}{2}} \prod_{i=1}^4 m(N_i) N_i \| u_i \|_{Y^0_T}.
\end{split}
\end{equation*}
This suffices.

\textbf{Case C } $N_4^* \gtrsim N$. In this case we infer from $L^4_{t,x}$-Strichartz estimates and the estimate $(m(N_i) N_i)^{-1} \lesssim N^{-(1-s)} N_i^{-s}$:
\begin{equation*}
\begin{split}
&\quad \big| \int_0^T \int_{\Gamma_4} \overline{M}_4(k_1,\ldots,k_4) \hat{u}_1(k_1) \ldots \hat{\overline{u}}(k_4) d\Gamma_4 ds \big| \\
&\lesssim_\varepsilon N_1^\varepsilon m(N_1^*) N_1^* m(N_3^*) N_3^* \prod_{i=1}^4 \| u_i \|_{Y^0_T} \\
&\lesssim N^{-2+} N_1^{0-} \prod_{i=1}^4 m(N_i) N_i \| u_i \|_{Y^0_T},
\end{split}
\end{equation*}
which suffices.
\end{proof}

\subsection{Estimate of the non-resonant contribution}

We turn to estimates for the boundary term. Under the dyadic frequency localization $N_1 \sim N_3 \gg N_3^*$ or $N_2 \sim N_4 \gg N_3^*$ we can use almost orthogonality to additionally localize the range of $k_1$ and $k_3$ to balls of size $N_3^*$.  Then we can argue like above to obtain the following:

\begin{lemma}
\label{lem:2dBoundaryTermSizeRegularity}
Let $N_1 \sim N_3 \gg N_3^*$ or $N_2 \sim N_4 \gg N_3^*$. Then it holds the size and regularity estimates:
\begin{equation*}
\big| \frac{M_4}{\Omega_4} \big| \lesssim 1.
\end{equation*}
\end{lemma}

We are ready to estimate the boundary term:
\begin{lemma}
\label{lem:BoundaryTerm2d}
Let $\frac{1}{2} < s \leq \frac{2}{3}$. Then it holds for $\delta = \delta(s)$
\begin{equation}
\label{eq:BoundaryTermEstimate2d}
| \Lambda_4(\tilde{\sigma}_4)(t) | \lesssim N^{-\delta} \| I u (t) \|_{H^1(\T^2_\lambda)}^4.
\end{equation}
\end{lemma}
\begin{proof}
We denote $u_i^* = P_{N_i} u$. Using the Fourier series argument, it suffices to show
\begin{equation*}
\big| \int_{\T^2_\lambda} u_1^*(x) \ldots u_4^*(x) dx \big| \lesssim N_1^{-\varepsilon} N^{-\delta} \prod_{i=1}^4 m(N_i) N_i \| P_{N_i} u_i \|_{L^2(\T^2_\lambda)}.
\end{equation*}
We apply H\"older's and Bernstein's inequality to find
\begin{equation*}
\begin{split}
\big| \int_{\T^2_\lambda} u_1^*(x) \ldots u_4^*(x) dx \big| &\lesssim \| P_{N_1^*} u_1^* \|_{L^2(\T^2_\lambda)} \| P_{N_2^*} u_2^* \|_{L^2(\T^2_\lambda)} \| P_{N_3^*} u_3^* \|_{L^\infty(\T^2_\lambda)} \| P_{N_4^*} u_4^* \|_{L^\infty(\T^2_\lambda)} \\
&\lesssim N_3^* N_4^* \prod_{i=1}^4 \| P_{N_i} u_i \|_{L^2(\T^2_\lambda)}.
\end{split}
\end{equation*}

Suppose that $N_1^* \sim N_2^* \gtrsim N$ and $N_3^* \lesssim N$.
Since $(m(N_i^*) N_i^*)^{-1} \lesssim N^{-(1-s)} (N_1^*)^{-s}$ for $i=1,2$, we find
\begin{equation*}
\big| \int_{\T^2_\lambda} u_1^* \ldots u_4^* dx \big| \lesssim N^{-2+} (N_1^*)^{0-} \prod_{i=1}^4 m(N_i) N_i \| P_{N_i} u_i \|_{L^2(\T^2_\lambda)}.
\end{equation*}

If $N_3^* \gtrsim N$ and $N_4^* \ll N$, note that for $\frac{1}{3} < s \leq \frac{2}{3}$, we have
\begin{equation*}
N_3^* \lesssim N^{0-} \prod_{i=1}^3 m(N_i^*) N_i^*,
\end{equation*}
which finishes this case.

Finally, suppose that $N_4^* \gtrsim N$. 
Since $m(N_i^*) N_i^* = \frac{N^{1-s}}{(N_i^*)^{1-s}} N_i^* = N^{1-s} (N_i^*)^s$, we have for $s \in (\frac{1}{2},\frac{2}{3})$
\begin{equation*}
(N_i^*)^{\frac{1}{2}} \lesssim_s m(N_i^*) N_i^* N^{-\frac{1}{3}} (N_i^*)^{0-}.
\end{equation*}
Consequently,
\begin{equation*}
(N_3^*)^{\frac{3}{2}} (N_4^*)^{\frac{1}{2}} \lesssim N^{-\frac{4}{3}} (N_1^*)^{0-} \prod_{i=1}^4 m(N_i^*) N_i^*.
\end{equation*}
The proof is complete.
\end{proof}

\subsection{Estimate for the non-resonant part}

In the following we show a growth bound for the non-resonant part:

\begin{proposition}
\label{prop:NonResonantEstimate2d}
For $\frac{1}{2} < s \leq \frac{2}{3}$ the following estimate holds:
\begin{equation}
\label{eq:NonResonantEstimate2d}
\big| \int_0^T \Lambda_6(\overline{M}_6) dt \big| \lesssim N^{-2+} \prod_{i=1}^6 m(N_i) N_i \| u_i \|_{Y^0_T}.
\end{equation}
\end{proposition}

Recall that 
\begin{equation*}
\overline{M}_6 = \pm i( \sum_{j=1}^4 X_j(\sigma_4) + X_j(\sigma_{4,2})).
\end{equation*}
Since we have already showed the size and regularity estimate for
\begin{equation*}
\sigma_{4,2} = \frac{M_4}{\Omega_4} \cdot \chi_{\tilde{A}},
\end{equation*}
it remains to check the size and regularity estimate for
\begin{equation*}
\sigma_4 = m(k_1) \ldots m(k_4),
\end{equation*}
which is straight-forward and recorded in the following:

\begin{lemma}
$\sigma_4$ satisfies the size and regularity estimate:
\begin{equation*}
|\sigma_4| \lesssim 1.
\end{equation*}
\end{lemma}

This way we can apply the Fourier series argument to $\sigma_4$ and $\sigma_{4,2}$ to reduce the estimate \eqref{eq:NonResonantEstimate2d} to the following:
\begin{lemma}
Let $\frac{1}{2} < s < \frac{2}{3}$ and $T \lesssim 1$. Suppose that $u_i$ is frequency localized at $N_i$ and $N_1 \sim N_2 \gtrsim N_3 \ldots \gtrsim N_6$ with $N_1 \gtrsim N$. Then the following estimate holds:
\begin{equation*}
\big| \int_0^T \int_{\T^2_\gamma} u_1(t,x) \ldots u_6(t,x) dx dt \big| \lesssim N^{-2+} N_1^{0-} \prod_{i=1}^6 m(N_i) N_i \| u_i \|_{Y^0_T}.
\end{equation*}
\end{lemma}
\begin{proof}
Suppose that $N_5 \ll N$. Then we can use four linear Strichartz estimates to find
\begin{equation*}
\begin{split}
\big| \int_0^T \int_{\T^2_\gamma} u_1(t,x) \ldots u_6(t,x) dx dt \big| &\lesssim \prod_{i=1}^4 \| u_i  \|_{L^4_{t,x}} \| u_5 \|_{L^\infty_{t,x}} \| u_6 \|_{L^\infty_{t,x}} \\
&\lesssim_\varepsilon (\lambda N_1 )^\varepsilon  N_5 N_6 \prod_{i=1}^6 \| u_i \|_{Y^0_T}.
\end{split}
\end{equation*}
Since $N_1 \sim N_2 \gtrsim N$, we have
\begin{equation*}
(m(N_1) N_1 m(N_2) N_2)^{-1} \lesssim N^{-2(1-s)} N_1^{-2s}
\end{equation*}
and thus,
\begin{equation*}
(\lambda N_1)^\varepsilon N_5 N_6 \prod_{i=1}^6 \| u_i \|_{Y^0_T} \lesssim N^{-2+} N_1^{0-} \prod_{i=1}^6 m(N_i) N_i \| u_i \|_{Y^0_T}.
\end{equation*}

Suppose that $N_5 \gtrsim N$ or $N_6 \gtrsim N$. We still find
\begin{equation*}
N_5 N_6 \lesssim m(N_5) N_5 m(N_6) N_6 \big( \prod_{i=1}^4 m(N_i) N_i \big) N_1^{0-} N^{-2+},
\end{equation*}
which allows us to conclude like above.
\end{proof}

\section{Conclusion of the proofs of Theorems \ref{thm:ImprovedGWPQuinticNLS} and \ref{thm:GWPNLS2d}}
\label{section:Conclusion}

\subsection{Proof of Theorem \ref{thm:ImprovedGWPQuinticNLS}}
We give a sketch of the argument.
The slow energy growth $N^{-3+}$ on the time interval $\lambda/N$ tells us that we can iterate the local well-posedness $N^{3-}$ times until the energy has grown significantly. This gives for the rescaled solution $u_\lambda$ on $\T_\lambda$ a time interval of existence of
\begin{equation*}
\frac{\lambda}{N} \cdot N^{3-} = \lambda N^{2-}.
\end{equation*}
Scaling back we obtain that the solution $u$ exists on a time scale
\begin{equation*}
T = N^{2-} \lambda^{-1} = N^{-\frac{1}{s}+3-}.
\end{equation*}
Consequently, for $s>\frac{1}{3}$, we obtain global well-posedness letting $N \to \infty$.

\bigskip

We turn to the details. To make the continuity argument transparent, we collect the previously established estimates and implicit constants.
Firstly, there is a universal constant $C_1 \geq 1$ such that
\begin{equation*}
\| I u_\lambda(t) \|_{H^1}^2 \leq C_1 (E_I^1(u_\lambda)(t) + \| u_\lambda(t) \|_{L^2}^2).
\end{equation*}
In the defocusing case this is clear, whereas in the focusing case this requires an application of the Gagliardo-Nirenberg inequality and the small mass assumption.
Secondly, there is a universal constant $C_2 \geq 1$ such that
\begin{equation*}
E_I^1(u_{0 \lambda}) \leq C_2 \| Iu_{0 \lambda} \|^2_{H^1}.
\end{equation*}
By Proposition \ref{prop:AuxiliaryWellposedness1d}, there is $\bar{C} > 0$ such that
\begin{equation*}
\| I u_\lambda \|_{Y^1_T} \leq \bar{C} 
\end{equation*}
for $T= \frac{\lambda}{N}$ provided that
\begin{equation*}
\| I u_{0 \lambda} \|^2_{H^1(\T_\lambda)} \leq c \ll 1.
\end{equation*}
Let $c_2 \leq c / (10 C_2 C_3)$. The key bounds to control the energy growth, which were proved in Propositions \ref{prop:EnergyGrowthResonant1d}, \ref{prop:DecayEstimateNonResonant1d}, and \ref{prop:BoundaryTerm1d}, are presently recalled:
\begin{equation}
\label{eq:EnergyGrowthEstimates}
\begin{split}
\big| \int_0^t \int  \overline{M}_6(k_1,\ldots,k_6) \hat{u}(s,k_1) \ldots \hat{\overline{u}}(s,k_6) d\Gamma_6 ds \big| &\lesssim N^{-3+} \| I u \|^6_{Y_{T}^1}, \\
\big| \int_0^t \int \overline{M}_{10}(k_1,\ldots,k_{10}) \hat{u}(s,k_1) \ldots \hat{\overline{u}}(s,k_{10}) d\Gamma_{10} ds \big| &\lesssim N^{-3+} \| I u \|^{10}_{Y_{T}^1}, \\
\big| \Lambda_6(\tilde{\sigma}_6)(t) \big| &\lesssim N^{-\delta} \| I u(t) \|^6_{H^1(\T_\lambda)}.
\end{split}
\end{equation}

We show the following:
\begin{proposition}
\label{prop:EnergyEstimate1d}
Suppose that
\begin{equation*}
\| I u_{0 \lambda} \|^2_{H^1(\T_\lambda)} \leq c_2.
\end{equation*}
Choose $N$ large enough such that
\begin{equation*}
N^{-\varepsilon} C_1 (\bar{C}^6 + \bar{C}^{10}) \leq C_1 C_2 c, \quad C_\delta N^{-\delta} (2(10 C_1 C_2 c)^6) \leq C_2 c.
\end{equation*}

Then, for $T = \lambda N^{2-2\varepsilon}$, the following estimate holds:
\begin{equation*}
\| I u_\lambda(t) \|_{H^1(\T_\lambda)}^2 \leq 10 C_2 C_3 c_2 \ll 1.
\end{equation*}
\end{proposition}
\begin{proof}
The proof is carried out via a continuity argument.
For times $0 \leq t \leq \frac{\lambda}{N}$ this follows from the auxiliary local well-posedness result. We turn to larger times: Write
\begin{equation*}
\begin{split}
E_I^1(u_\lambda(t)) &= E_I^1(u_{0 \lambda}) - \big( \Lambda_6(\tilde{\sigma}_6(t)) - \Lambda_6(\tilde{\sigma}_6)(0)) \\
&\quad + \int_0^t \Lambda_6(\overline{M}_6) ds + \int_0^t \Lambda_{10}(\overline{M}_{10}) ds.
\end{split}
\end{equation*}
Suppose that $\| Iu_\lambda(t) \|_{H^1(\T_\lambda)}^2 \leq 10 C_1 C_2 c$ and $0 \leq t \leq N^{2-} \lambda$,  $M \in \N$ such that $T \in [\frac{(M-1) \lambda}{N}, \frac{M \lambda}{N}]$. By the above display and the bounds \eqref{eq:EnergyGrowthEstimates}, we obtain:
\begin{equation*}
\begin{split}
\| I u_\lambda(t) \|_{H^1}^2 &\leq C_1 ( E_I^1 u_\lambda(t) + \| u_{0 \lambda} \|_{L^2}^2 ) \\
&\leq C_1 (E_I^1(u_{0 \lambda}) - \big( \Lambda_6(\tilde{\sigma}_6)(t) - \Lambda_6(\tilde{\sigma})(0) \big) \\
&\quad + \sum_{k=0}^{M-2} \int_{[k \lambda/N, (k+1) \lambda/N]} \big( \Lambda_6(\overline{M}_6) + \Lambda_{10}(\overline{M}_{10}) \big) ds \\
&\quad + \int_{[\frac{(M-1)\lambda}{N},T]} \Lambda_6(\overline{M}_6) + \Lambda_{10}(\overline{M}_{10}) ds \\
&\leq C_1 (E_I^1(u_{0 \lambda}) + C_\delta N^{-\delta} \big( \| I u_\lambda(t) \|_{H^1}^6 + \| I u_\lambda(0) \|_{H^1}^6 \big) + M N^{-3+} ( \bar{C}^6 + \bar{C}^{10}) \big) \\
&\leq C_1 ( C_2 \| I u_{0 \lambda} \|_{H^1}^2 + M N^{-3+} (\bar{C}^6 + \bar{C}^{10}) + C_\delta N^{-\delta} (2 (10 C_1 C_2 c)^6 ) \\
&\leq 5 C_1 C_2 c.
\end{split}
\end{equation*}
Hence, we have small $H^1$-norms on a timescale $N^{2-} \lambda$.
\end{proof}

\begin{proof}[Proof~of~Theorem~\ref{thm:ImprovedGWPQuinticNLS}]

Suppose that $u_0 \in H^s(\T)$ for $s>\frac{1}{3}$ with $\| u_0 \|_{L^2(\T)} \leq c \ll 1$. We rescale $u_{0\lambda}$ to the large torus $\T_{\lambda}$ and obtain with the choice $\lambda = N^{\frac{1-s}{s}+\varepsilon}$:
\begin{equation*}
\begin{split}
\| I u_{0 \lambda} \|^2_{H^1(\T_\lambda)} &\leq 2( \| u_{0 \lambda} \|^2_{L^2(\T_\lambda)} + N^{2(1-s)} \| u_{0 \lambda} \|^2_{\dot{H}^s(\T_\lambda)} ) \\
&\leq 2 ( \| u_0 \|^2_{L^2(\T)} + N^{2(1-s)} \lambda^{-2s} \| u_0 \|^2_{\dot{H}^s(\T)} ) \\
&\leq 2 ( \| u_0 \|^2_{L^2(\T)} + N^{-2s \varepsilon} \| u_0 \|^2_{\dot{H}^s(\T)} ).
\end{split}
\end{equation*}
Now we can choose $N=N( \| u_0 \|_{\dot{H}^s}, \| u_0 \|_{L^2(\T)})$ large enough such that
\begin{equation*}
\| I u_{0 \lambda} \|^2_{H^1(\T_\lambda)} \leq c_2 \ll 1
\end{equation*}
with $c_2$ like in Proposition \ref{prop:EnergyEstimate1d}.
By Proposition \ref{prop:EnergyEstimate1d} the $I$-Cauchy problem \eqref{eq:ICauchyProblem} can be solved for times $T \sim \lambda N^{2-}$.

Scaling back, we find that the solution on $\T$ exists for times
\begin{equation*}
T^* \sim \lambda^{-1} N^{2-} = N^{3-\frac{1}{s}-}.
\end{equation*}
Hence, for $s>\frac{1}{3}$ and letting $N \to \infty$, we obtain global well-posedness.

\end{proof} 

\subsection{Proof of Theorem \ref{thm:GWPNLS2d}}

In this section we finish the proof of Theorem \ref{thm:GWPNLS2d}. We collect the proved estimates and formulae for convenience.

\medskip

Starting with local solutions to
\begin{equation}
\label{eq:SolutionsT2}
\left\{ \begin{array}{cl}
i \partial_t u + \Delta u &= \pm |u|^2 u, \quad (t,x) \in \R \times \T^2_\gamma, \\
u(0) &= u_0 \in H^s(\T_{\underline{\gamma}}^2)
\end{array} \right.
\end{equation}
for $\frac{1}{2} < s \leq \frac{2}{3}$, we rescale to tori $\T^2_\lambda$ with large period and damp the high frequencies to consider the $I$-system
\begin{equation}
\label{eq:ISystemConclusion}
\left\{ \begin{array}{cl}
i \partial_t I u_\lambda + \Delta I u_\lambda &= \pm I(|u_\lambda|^2 u_\lambda), \quad (t,x) \in \R \times \T^2_\lambda, \\
Iu_\lambda(0) &= I u_{\lambda 0} \in H^1(\T^2_\lambda). 
\end{array} \right.
\end{equation}
With our choice of $\lambda = N^{\frac{1-s}{s}}$ we accomplish
\begin{equation*}
\| I u_\lambda(0) \|_{H^1(\T^2_\lambda)} \lesssim \| u_{\lambda 0 } \|_{H^s(\T^2_\lambda)},
\end{equation*}
and we have local well-posedness of \eqref{eq:ISystemConclusion} on times $T = \lambda^{-\delta}$.

We use the $I$-energy as an almost conserved quantity:
\begin{equation*}
E_I^1(u_\lambda)(t) = E_I^1(u_{\lambda 0}) - \big[ \Lambda_4(\tilde{\sigma}_4)(t) - \Lambda_4(\tilde{\sigma}_4)(0) \big] + \int_0^t \big( \Lambda_4(\overline{M}_4) + \Lambda_6(\overline{M}_6) \big) ds.
\end{equation*}
The relevant growth bounds for $\frac{1}{2} < s \leq \frac{2}{3}$ and $0 \leq T \leq \lambda^{-\delta}$ are given by Proposition \ref{prop:EnergyGrowthResonant2d} and \ref{prop:NonResonantEstimate2d}:
\begin{equation*}
\big| \int_0^T \Lambda_4(\overline{M}_4) ds \big| \lesssim N^{-1+} \lambda^{-\frac{1}{2}} \| I u_{\lambda} \|^4_{Y^1_T}, \quad \big| \int_0^T \Lambda_6(\overline{M}_6) ds \big| \lesssim N^{-2+} \| I u_{\lambda} \|^6_{Y^1_T}.
\end{equation*}
Note that the estimate for the resonant part is clearly inferior to the estimate for the non-resonant part.
The estimate for the boundary term for $\frac{1}{2} < s \leq \frac{2}{3}$ is given by some $C_\delta$ and $N^{-\delta}$
\begin{equation*}
\big| \Lambda_4(\tilde{\sigma}_4)(t) \big| \leq C_\delta N^{-\delta} \| I u_{\lambda}(t) \|^4_{H^1(\T^2_\lambda)}.
\end{equation*}

\medskip

In the defocusing case we have
\begin{equation}
\label{eq:EnergyCoercivity}
\| I u_{\lambda} \|^2_{\dot{H}^1(\T^2_\lambda)} \leq 2 E_I^1(u(t)), \quad \| Iu_{\lambda} (t) \|^2_{L^2(\T^2_\lambda)} \leq \| u_{0 \lambda} \|^2_{L^2(\T^2_\lambda)}.
\end{equation}
For the first estimate to hold in the focusing case, we require a smallness condition on $\| u_{0 \lambda} \|_{L^2(\T_\lambda)} = \| u_{0 \lambda} \|_{L^2(\T^2_\gamma)} \ll 1$. Then follows from the Gagliardo--Nirenberg--Ladyzhenskaya inequality
\begin{equation*}
\| I u_{\lambda}(t) \|_{L^4}^4 \ll \| I u_{\lambda}(t) \|_{\dot{H}^1}^2.
\end{equation*}
Consequently, we obtain as well in the focusing case:
\begin{equation*}
\| I u_{\lambda}(t) \|_{\dot{H}^1}^2 \leq 2 \big( \| I u_{\lambda} (t) \|_{\dot{H}^1}^2 - \int_{\T^2_\gamma} \frac{|I u_{\lambda} |^4}{4} dx \big).
\end{equation*}

We choose $N$ large enough depending on the following.
Firstly, note that there is a universal constant $C_2$ such that
\begin{equation*}
E_I^1(u_{0 \lambda}) \leq C_2 \| I u_{0 \lambda} \|_{H^1}^2.
\end{equation*}
Given $u_0$, by the auxiliary local well-posedness we find $\bar{C}$ such that
\begin{equation*}
\| I v_{\lambda} \|_{Y^1_T} \leq \bar{C} \| I v_{0 \lambda} \|_{H^1}
\end{equation*}
for $T \leq \lambda^{-\delta}$ as long as
\begin{equation*}
\| I v_{0 \lambda} \|^2_{H^1} \leq 10 C_2 \| I u_{0 \lambda} \|^2_{H^1}.
\end{equation*}
Now we choose $N \gg 1$ such that
\begin{equation*}
N^{-\varepsilon} \big( \bar{C}^4 (5 C_2)^4 \| I u_{0 \lambda} \|_{H_\lambda^1}^4 + \bar{C}^6 (5C_2)^6 \| I u_{0 \lambda} \|_{H_\lambda^1}^6) \leq \| I u_0 \|_{H_{\lambda}^1}^2.
\end{equation*}
Secondly, we choose $N \gg 1$ possibly larger such that
\begin{equation*}
26 C_2 C_\delta N^{-\delta} \| I u_{0 \lambda} \|_{H_\lambda^1}^4 \leq \| I u_{0 \lambda} \|_{H_\lambda^1}^2.
\end{equation*}

With this choice of $N$, we show the following proposition:
\begin{proposition}
\label{prop:LongTimeExistence}
Suppose that $\| u_0 \|_{L^2(\T_\gamma)} \ll 1$ in the focusing case such that \eqref{eq:EnergyCoercivity} holds. There is $N=N(\|u_0 \|_{H^s})$ such that for $0 \leq t \leq T \lesssim N^{1-} \lambda^{\frac{1}{2}}$ the following estimate holds for the solution to \eqref{eq:ISystemConclusion}:
\begin{equation}
\label{eq:IBound}
\| I u_\lambda(t) \|^2_{H^1(\T^2_\lambda)} \leq 5 C_2 \| I u_{0 \lambda} \|_{H^1}^2.
\end{equation}
\end{proposition}
\begin{proof}
We use a continuity argument. The claim holds for $|t| \leq \lambda^{-\delta}$ by the auxiliary local well-posedness. Next suppose \eqref{eq:IBound} holds up to $T \ll N^{1-} \lambda^{\frac{1}{2}}$. Write $T= M \lambda^{-\delta}$ for $M \in \N$. Then we estimate
\begin{equation*}
\begin{split}
E_I^1(u_{\lambda}(T)) &= E_I^1(u_{0 \lambda}) - \big[ \Lambda_4(\tilde{\sigma}_4)(t) - \Lambda_4(\tilde{\sigma}_4)(0) \big] + \int_0^{\lambda^{-\delta}} \big( \Lambda_4(\overline{M}_4) + \Lambda_6(\overline{M}_6) \big)   \\
&\quad + \int_{\lambda^{-\delta}}^{2 \lambda^{-\delta}} \big( \Lambda_4(\overline{M}_4) + \Lambda_6(\overline{M}_6) \big) + \ldots + \int_{(M-1) \lambda^{-\delta}}^{M \lambda^{-\delta}} \big( \Lambda_4(\overline{M}_4) + \Lambda_6(\overline{M}_6) \big).
\end{split}
\end{equation*}
By the bootstrap assumption we have the following estimate for the boundary term:
\begin{equation*}
C_\delta N^{-\delta} (\| I u_\lambda(t) \|^4_{H^1} +\| I u_{0 \lambda} \|_{H^1}^4 ) \leq C_\delta N^{-\delta} 26 C_2 \| I u_{0 \lambda} \|^4_{H^1} \leq \| I u_{0 \lambda} \|_{H^1}^2.
\end{equation*}
Consequently,
\begin{equation*}
\begin{split}
E_I^1(u_\lambda(T)) &\leq E_I^1(u_{0 \lambda}) + \| I u_{0 \lambda} \|_{H^1}^2 \\
&\quad + M \lambda^{-\delta} N^{-1+\varepsilon} \lambda^{-\frac{1}{2}} N^{-\varepsilon} \tilde{C} ( \bar{C}^4 (5 C_2)^4 \| I u_{0 \lambda} \|_{H^1}^4 + \bar{C}^6 (5 C_2)^6 \| I u_{0 \lambda} \|^6_{H^1} ).
\end{split}
\end{equation*}
The final term can be bounded by $\| I u_{0 \lambda} \|_{H^1}^2$ and thus,
\begin{equation*}
\| I u_\lambda(T) \|_{H^1}^2 \leq E_I^1(u_\lambda(T)) + \| I u_\lambda(T) \|_{L^2}^2 \leq 4 \| I u_{0 \lambda} \|_{H^1}^2.
\end{equation*}
Hence, we can continue the bootstrap assumption until $T \sim N^{1-} \lambda^{\frac{1}{2}}$.
\end{proof}

With the above proposition at hand, the conclusion is now a matter of rescaling:
\begin{proof}[Proof~of~Theorem~\ref{thm:GWPNLS2d}]
We choose $N$ large enough to have \eqref{eq:IBound} from Proposition \ref{prop:LongTimeExistence}. Reversing the scaling we find that the solution to \eqref{eq:SolutionsT2} exists on a time
\begin{equation*}
T \sim \lambda^{-\frac{3}{2}} N^{1-} \sim N^{-\frac{3}{2s}+\frac{5}{2}-}.
\end{equation*}
The exponent is positive provided that $s > \frac{3}{5}$. Hence, taking $N$ to infinity for $s>\frac{3}{5}$ establishes global well-posedness.
\end{proof}

\section*{Acknowledgement}

Financial support by Korea Institute for Advanced Study, grant No.
MG093901 is gratefully acknowledged.


\begin{thebibliography}{10}

\bibitem{Bourgain1993A}
J.~Bourgain.
\newblock Fourier transform restriction phenomena for certain lattice subsets
  and applications to nonlinear evolution equations. {I}. {S}chr\"{o}dinger
  equations.
\newblock {\em Geom. Funct. Anal.}, 3(2):107--156, 1993.

\bibitem{Bourgain2004}
Jean Bourgain.
\newblock A remark on normal forms and the ``{$I$}-method'' for periodic {NLS}.
\newblock {\em J. Anal. Math.}, 94:125--157, 2004.

\bibitem{BourgainDemeter2015}
Jean Bourgain and Ciprian Demeter.
\newblock The proof of the {$l^2$} decoupling conjecture.
\newblock {\em Ann. of Math. (2)}, 182(1):351--389, 2015.

\bibitem{BurqGerardTzvetkov2004}
N.~Burq, P.~G\'{e}rard, and N.~Tzvetkov.
\newblock Strichartz inequalities and the nonlinear {S}chr\"{o}dinger equation
  on compact manifolds.
\newblock {\em Amer. J. Math.}, 126(3):569--605, 2004.

\bibitem{CollianderKeelStaffilaniTakaokaTao2002}
J.~Colliander, M.~Keel, G.~Staffilani, H.~Takaoka, and T.~Tao.
\newblock Almost conservation laws and global rough solutions to a nonlinear
  {S}chr\"{o}dinger equation.
\newblock {\em Math. Res. Lett.}, 9(5-6):659--682, 2002.

\bibitem{CollianderKeelStaffilaniTakaokaTao2003}
J.~Colliander, M.~Keel, G.~Staffilani, H.~Takaoka, and T.~Tao.
\newblock Sharp global well-posedness for {K}d{V} and modified {K}d{V} on
  {$\Bbb R$} and {$\Bbb T$}.
\newblock {\em J. Amer. Math. Soc.}, 16(3):705--749, 2003.

\bibitem{CollianderKeelStaffilaniTakaokaTao2004}
J.~Colliander, M.~Keel, G.~Staffilani, H.~Takaoka, and T.~Tao.
\newblock Multilinear estimates for periodic {K}d{V} equations, and
  applications.
\newblock {\em J. Funct. Anal.}, 211(1):173--218, 2004.

\bibitem{DeSilvaPavlovicStaffilaniTzirakis2007}
Daniela De~Silva, Nata\v{s}a Pavlovi\'{c}, Gigliola Staffilani, and Nikolaos
  Tzirakis.
\newblock Global well-posedness for a periodic nonlinear {S}chr\"{o}dinger
  equation in 1{D} and 2{D}.
\newblock {\em Discrete Contin. Dyn. Syst.}, 19(1):37--65, 2007.

\bibitem{FanStaffilaniWangWilson2018}
Chenjie Fan, Gigliola Staffilani, Hong Wang, and Bobby Wilson.
\newblock On a bilinear {S}trichartz estimate on irrational tori.
\newblock {\em Anal. PDE}, 11(4):919--944, 2018.

\bibitem{GuoLiYung2021}
Shaoming {Guo}, Zane~Kun {Li}, and Po-Lam {Yung}.
\newblock {Improved discrete restriction for the parabola}.
\newblock {\em arXiv e-prints}, page arXiv:2103.09795, March 2021.

\bibitem{HadacHerrKoch2009}
Martin Hadac, Sebastian Herr, and Herbert Koch.
\newblock Well-posedness and scattering for the {KP}-{II} equation in a
  critical space.
\newblock {\em Ann. Inst. H. Poincar\'{e} C Anal. Non Lin\'{e}aire},
  26(3):917--941, 2009.

\bibitem{HadacHerrKoch2010}
Martin Hadac, Sebastian Herr, and Herbert Koch.
\newblock Erratum to ``{W}ell-posedness and scattering for the {KP}-{II}
  equation in a critical space'' [{A}nn. {I}. {H}. {P}oincar\'{e}---{AN} 26 (3)
  (2009) 917--941] [mr2526409].
\newblock {\em Ann. Inst. H. Poincar\'{e} C Anal. Non Lin\'{e}aire},
  27(3):971--972, 2010.

\bibitem{Hani2012}
Zaher Hani.
\newblock A bilinear oscillatory integral estimate and bilinear refinements to
  {S}trichartz estimates on closed manifolds.
\newblock {\em Anal. PDE}, 5(2):339--363, 2012.

\bibitem{HerrKwak2023}
Sebastian {Herr} and Beomjong {Kwak}.
\newblock {Strichartz estimates and global well-posedness of the cubic NLS on
  $\mathbb{T}^{2}$}.
\newblock {\em arXiv e-prints}, page arXiv:2309.14275, September 2023.

\bibitem{HerrTataruTzvetkov2011}
Sebastian Herr, Daniel Tataru, and Nikolay Tzvetkov.
\newblock Global well-posedness of the energy-critical nonlinear
  {S}chr\"{o}dinger equation with small initial data in {$H^1(\Bbb T^3)$}.
\newblock {\em Duke Math. J.}, 159(2):329--349, 2011.

\bibitem{Kishimoto2014}
Nobu Kishimoto.
\newblock Remark on the periodic mass critical nonlinear {S}chr\"{o}dinger
  equation.
\newblock {\em Proc. Amer. Math. Soc.}, 142(8):2649--2660, 2014.

\bibitem{KochTataru2005}
Herbert Koch and Daniel Tataru.
\newblock Dispersive estimates for principally normal pseudodifferential
  operators.
\newblock {\em Comm. Pure Appl. Math.}, 58(2):217--284, 2005.

\bibitem{KochTataru2007}
Herbert Koch and Daniel Tataru.
\newblock A priori bounds for the 1{D} cubic {NLS} in negative {S}obolev
  spaces.
\newblock {\em Int. Math. Res. Not. IMRN}, (16):Art. ID rnm053, 36, 2007.

\bibitem{LiWuXu2011}
Yongsheng Li, Yifei Wu, and Guixiang Xu.
\newblock Global well-posedness for the mass-critical nonlinear
  {S}chr\"{o}dinger equation on {$\Bbb T$}.
\newblock {\em J. Differential Equations}, 250(6):2715--2736, 2011.

\bibitem{MoyuaVega2008}
A.~Moyua and L.~Vega.
\newblock Bounds for the maximal function associated to periodic solutions of
  one-dimensional dispersive equations.
\newblock {\em Bull. Lond. Math. Soc.}, 40(1):117--128, 2008.

\bibitem{Schippa2023}
Robert {Schippa}.
\newblock {Refinements of Strichartz estimates on tori and applications}.
\newblock {\em arXiv e-prints}, page arXiv:2312.06333, December 2023.

\bibitem{TakaokaTzvetkov2001}
H.~Takaoka and N.~Tzvetkov.
\newblock On 2{D} nonlinear {S}chr\"{o}dinger equations with data on {${\Bbb
  R}\times\Bbb T$}.
\newblock {\em J. Funct. Anal.}, 182(2):427--442, 2001.

\bibitem{Tao2006}
Terence Tao.
\newblock {\em Nonlinear dispersive equations}, volume 106 of {\em CBMS
  Regional Conference Series in Mathematics}.
\newblock Conference Board of the Mathematical Sciences, Washington, DC; by the
  American Mathematical Society, Providence, RI, 2006.
\newblock Local and global analysis.

\bibitem{Tzirakis2010}
Nikolaos Tzirakis.
\newblock Errata to ``{The} {Cauchy} problem for the semi-linear quintic
  {Schr{\"o}dinger} equation in 1d''.
\newblock {\em Differ. Integral Equ.}, 23(3-4):399--400, 2010.

\bibitem{Wiener1979}
Norbert Wiener.
\newblock {\em Collected works with commentaries. {V}ol. {II}}, volume~15 of
  {\em Mathematicians of Our Time}.
\newblock MIT Press, Cambridge, Mass.-London, 1979.
\newblock Generalized harmonic analysis and Tauberian theory; classical
  harmonic and complex analysis, Edited by Pesi Rustom Masani.

\end{thebibliography}
\end{document}